\documentclass[10pt,a4paper,reqno,twoside]{amsart}

\usepackage{enumerate, mathtools}
\usepackage{amsmath,amssymb,amscd,bezier}
\usepackage{amsfonts}
\usepackage[english]{babel}
\usepackage[latin1]{inputenc}
\usepackage[T1]{fontenc}
\usepackage{mathrsfs}
\usepackage{color}

\theoremstyle{plain}
\newtheorem{theorem}{Theorem}
\newtheorem{proposition}{Proposition}
\newtheorem{lemma}{Lemma}
\newtheorem{corollary}{Corollary}

\theoremstyle{definition}

\newtheorem{notation}{Notation}

\theoremstyle{remark}
\newtheorem{remark}{Remark}
\newtheorem{example}{Example}

\allowdisplaybreaks

\newcommand{\R}{\mathbb R}
\newcommand{\Z}{\mathbb Z}

\renewcommand{\phi}{\varphi}

\newcommand{\N}{\mathbb N}
\newcommand{\supp}{\mbox{supp\,}}
\newcommand{\lin}{{\mathrm{\,lin\,}}}

\newcommand{\xii}{{\abs{\xi}}}

\DeclarePairedDelimiter{\abs}{\lvert}{\rvert}
\DeclarePairedDelimiter{\norma}{\lVert}{\rVert}

\newcommand{\norm}[1]{\lVert#1\rVert}
\DeclarePairedDelimiter{\ceil}{\lceil}{\rceil}

\newcommand{\sinc}{{\mathrm{\,sinc\,}}}

\newcommand{\defeq}{{\mathrm{\,:=\,}}}
\renewcommand{\doteq}{\defeq}

\renewcommand{\mathfrak}{\mathscr}

\newcommand{\eps}{\varepsilon}

\def\<#1\>{\left\langle#1\right\rangle }

\title[The critical exponent for nonlinear damped $\sigma$-evolution equations]{The critical exponent for nonlinear \\
damped $\sigma$-evolution equations}

\author{Marcello D'Abbicco, Marcelo Rempel Ebert}

\address{Marcello D'Abbicco, Department of Mathematics, University of Bari, Via E. Orabona 4 - 70125 BARI - ITALY; marcello.dabbicco@uniba.it}

\address{Marcelo Rempel Ebert, Departamento de Computa\c{c}\~ao e Matem\'atica, Universidade de S\~ao Paulo, Ribeir\~ao Preto, SP, 14040-901, Brasil}

\begin{document}

\begin{abstract}
In this paper, we derive suitable optimal $L^p-L^q$ decay estimates, $1\leq p\leq q\leq \infty$, for the solutions to the $\sigma$-evolution equation, $\sigma>1$, with structural damping and power nonlinearity~$|u|^{1+\alpha}$ or~$|u_t|^{1+\alpha}$,
\[ u_{tt}+(-\Delta)^\sigma u +(-\Delta)^\theta u_t=\begin{cases}
|u|^{1+\alpha}, \\
|u_t|^{1+\alpha},
\end{cases}\]
where~$t\geq0$ and~$x\in\R^n$. Using these estimates, we can solve the problem of finding the critical exponents for the two nonlinear problems above in the so-called non-effective case, $\theta\in(\sigma/2,\sigma]$. This latter is more difficult than the effective case~$\theta\in[0,\sigma/2)$, since the asymptotic profile of the solution involves a diffusive component and an oscillating one. The novel idea in this paper consists in treating separately the two components to neglect the loss of decay rate created by the interplay of the two components. We deal with the oscillating component, by localizing the low frequencies, where oscillations appear, in the extended phase space. This strategy allows us to recover a quasi-scaling property which replaces the lack of homogeneity of the equation.
\end{abstract}



\maketitle

\section{Introduction}

In this paper we study the critical exponent of small data global-in-time solutions for the forward Cauchy problem for a $\sigma$-evolution equation with a so-called structural damping and with a power nonlinearity~$f(u)$, in particular,
\begin{equation}
\label{eq:CPu}
\begin{cases}
u_{tt}+(-\Delta)^\sigma u +(-\Delta)^\theta u_t=|u|^{1+\alpha}, \quad x\in \R^n, \, t\in\R_+,\\
u(0,x)=0, \\
u_t(0,x)=u_1(x),
\end{cases}
\end{equation}
or with a power nonlinearity~$f(u_t)$, in particular,
\begin{equation}
\label{eq:CPut}
\begin{cases}
u_{tt}+(-\Delta)^\sigma u +(-\Delta)^\theta u_t=|u_t|^{1+\alpha}, \quad x\in \R^n, \, t\in\R_+,\\
u(0,x)=0, \\
u_t(0,x)=u_1(x).
\end{cases}
\end{equation}
The term~$(-\Delta)^\sigma$ stands for higher powers of the Laplace operator, which may possibly be non-integer. In general, it is assumed~$\sigma>1$, real. In the non-integer case, $(-\Delta)^\sigma f= \mathfrak{F}^{-1}(\xii^{2\sigma}\hat f)$, for~$f$ in a suitable function space. Equations whose ``principal part'' is $u_{tt}+(-\Delta)^\sigma u=0$, like the plate equation which is attained for~$\sigma=2$, are called $\sigma$-evolution equations in the sense of Petrowsky (see~\cite{ER}), since their symbols $\tau^2 + |\xi|^{2\sigma}$ have only pure imaginary, distinct, roots $\tau=\pm i|\xi|^{\sigma}$ for all $\xi \neq 0$. The set of 1-evolution operators coincides with the set of strictly hyperbolic operators.

The term~$(-\Delta)^\theta u_t$ represents a damping, a term whose action dissipates the energy
\[ E(t)= \frac12\,\|u_t(t,\cdot)\|_{L^2}^2 + \frac12\,\|(-\Delta)^{\frac\sigma2}u(t,\cdot)\|_{L^2}^2 \]
of the corresponding linear equation
\begin{equation}
\label{eq:CPlin}
\begin{cases}
u_{tt}+(-\Delta)^\sigma u +(-\Delta)^\theta u_t=0, \quad x\in \R^n, \, t\in\R_+,\\
u(0,x)=0, \\
u_t(0,x)=u_1(x).
\end{cases}
\end{equation}
Indeed,
\[ E'(t) = - \|(-\Delta)^{\frac\theta2}u_t(t,\cdot)\|_{L^2}^2\leq0. \]
The range~$\theta\in[0,\sigma]$ is of interest, with~$\theta=0$ representing a classical damping, also called exterior or weak damping, and~$\theta\in(0,\sigma]$ represents a structural damping, also called strong damping.

The nonlinearity may have several shapes, but we are mainly interested in the model cases~$f(u)=|u|^{1+\alpha}$ or~$f(u_t)=|u_t|^{1+\alpha}$, for~$\alpha>0$. The important information is that the nonlinearity is without sign, i.e., it is not of type~$\pm u|u|^\alpha$, so it is in general a perturbation which may create blow-up in finite time.

In recent years there has been a growing attention to find the so-called \emph{critical exponent}~$\bar\alpha$ for problems~\eqref{eq:CPu} and~\eqref{eq:CPut}. By critical exponent, we mean that global solutions to~\eqref{eq:CPu} or~\eqref{eq:CPut} exist for sufficiently small data, when~$\alpha>\bar\alpha$, whereas solutions cannot exist globally, in general, when~$\alpha\in(1,\bar\alpha)$. The critical case~$\alpha=\bar\alpha$ sometimes belongs to the nonexistence interval, and sometimes to the existence interval.

As expected in the setting of small data solutions for nonlinear problem, the critical exponent~$\bar\alpha$ is mainly determined by the profile of the decay rate of the solution in suitable norms. In particular, it is in general relevant the vanishing speed as~$t\to\infty$ of the norm~$\|u(t,\cdot)\|_{L^{1+\alpha}}$. However, the action of the damping term~$(-\Delta)^\theta u_t$ deeply modifies the asymptotic profile of the solution as~$t\to\infty$. According to a classification introduced for more general problems in~\cite{DAE16} (see also~\cite{W06,W07} for the original definition with classical damping and time-dependent coefficients), we say that the damping is \emph{effective} when~$\theta\in[0,\sigma/2)$ and \emph{non-effective} when~$\theta\in(\sigma/2,\sigma]$.

\subsection*{The effective case: the easier case, already solved}

In the effective case~$\theta\in[0,\sigma/2)$, a diffusion phenomenon appears which make the asymptotic profile of the solution to~\eqref{eq:CPlin} to be determined by the solution to a diffusive problem. More precisely, $u=u_++u_-$, where
\begin{align*}
u_+(t,\cdot)
    & \sim v_+(t,\cdot) =e^{-t(-\Delta)^{\sigma-\theta}} I_{2\theta} u_1,\\
u_-(t,\cdot)
    & \sim v_-(t,\cdot) = e^{-t(-\Delta)^\theta} I_{2\theta} u_1,
\end{align*}
in the following sense:
\begin{align*}
\|(u_+-v_+)(t,\cdot)\|_{L^p}
    & =\textit{o}\big( t^{-\frac{n}{2(\sigma-\theta)}\left(1-\frac1p\right)+\frac\theta{\sigma-\theta}} \big)\|u_1\|_{L^1}, \\
\|(u_--v_-)(t,\cdot)\|_{L^p}
    & =\textit{o}\big( t^{-\frac{n}{2\theta}\left(1-\frac1p\right)+1} \big)\|u_1\|_{L^1}.
\end{align*}
Here and in the following, $I_af=(-\Delta)^{-\frac{a}2}f=\mathfrak{F}^{-1}(\xii^{-a}\hat f)$ denotes the Riesz potential, and~$v=e^{-t(-\Delta)^b}\phi$ means that~$v$ is the solution to
\[ v_t +(-\Delta)^bv=0, \quad v(0,x)=\phi(x). \]
As a consequence, the optimal decay estimate
\begin{equation}\label{eq:optimaleff}
\|\partial_t^ju(t,\cdot)\|_{L^p}\leq \begin{cases}
C\,t^{-j-\frac{n}{2(\sigma-\theta)}\left(1-\frac1p\right)+\frac\theta{\sigma-\theta}} & \text{if~$n(1-1/p)\geq 2\theta$} \\
C\,t^{-j-\frac{n}{2\theta}\left(1-\frac1p\right)+1}& \text{if~$n(1-1/p)\leq 2\theta$,}
\end{cases}
\end{equation}
follows, for~$j=0,1$ and~$p\in[1,\infty]$ (see~\cite{DAE14JDE, K00}).

This strong analogy between the evolution equation and the simpler diffusive problem, which also reflects for the corresponding nonlinear problems, allowed in recent years to determinate that the critical exponents in the effective case for~\eqref{eq:CPu} and~\eqref{eq:CPut} (see~\cite{DAE14NA, DAE17NA}), respectively, are given by the values
\[ \bar\alpha = \frac{2\sigma}{n-2\theta}, \quad\text{and, respectively,} \qquad \bar\alpha=\frac{2\theta}n. \]
These values correspond to set~$p=1+\bar\alpha$, in such a way that $p$ is the unique value such that $p$ times the decay rate in~\eqref{eq:optimaleff} gives~$1$.

The limit case~$\theta=\sigma/2$ is in general easier, since its asymptotic profile is simpler, and the same results above hold (see~\cite{DA14proc}; see also~\cite{DAEL, DKR}).

The case of wave equation with classical damping, i.e. $\sigma=1$ and $\theta=0$, has been first investigated by A. Matsumura~\cite{M76}, who determined the existence of small data global solutions to~\eqref{eq:CPu} in the supercritical case~$\alpha>2/n$ in space dimension~$n=1,2$. Only later on, the result has been extended to any space dimension~$n\geq3$ by G. Todorova and B. Yordanov~\cite{TY01} (see also~\cite{IT05}), with the nonexistence counterpart proved in~\cite{Z01} and the diffusion phenomenon showed in~\cite{HM, HT, MN03, N03}. The critical case for more general nonlinearities has been recently discussed in~\cite{EGR}.

\subsection*{The non-effective case: the more difficult case}

The situation is completely different in the non-effective case, since oscillations appear in the asymptotic profile of the solution. This case corresponds to the ``damped oscillations'' case for the ordinary differential equation of the harmonic oscillator:
\[ y'' + \omega^2 y + 2by'=0, \]
where~$b,\omega>0$. If~$b\in(0, \omega)$ then the friction action damps the oscillations, namely,
\[ y=e^{-bt}\Big(C_1\sin t\sqrt{\omega^2-b^2} +C_2\cos t\sqrt{\omega^2-b^2} \Big),\]
without destroying them. On the converse, if~$b>\omega$ then
\[ y=C_1e^{(-b+\sqrt{b^2-\omega^2})t}+C_2e^{(-b-\sqrt{b^2-\omega^2})t},\]
and oscillations disappear (which corresponds to the case of effective damping in~\eqref{eq:CPlin}).

Describing the asymptotic profile of the solution to~\eqref{eq:CPlin} is more difficult for $\theta\in(\sigma/2,\sigma]$. Applying Fourier transform to~\eqref{eq:CPlin} with respect to the~$x$ variable and denoting~$\hat u(t,\xi)=\mathfrak{F}(u(t,\cdot))$, we get
\[ \hat u(t,\xi) = t\,e^{-t\xii^{2\theta}/2}\,\sinc (\omega t)\,\hat u_1(\xi),\]
for sufficiently small~$\xii$, where~$\sinc\rho=\rho^{-1}\sin\rho$ is the cardinal sin function, and
\[ \omega = \xii^\sigma\,\sqrt{1-\xii^{4\theta-2\sigma}/4}.  \]
(We mention that the asymptotic profile and the decay rate structure is very different if~$\theta>\sigma$ and new effects appear: we will not investigate this case in this paper, but we address the interested reader to~\cite{GGH}).

Due to~$\omega\sim\xii^\sigma$ as~$\xi\to0$, roughly speaking, we may say that the asymptotic profile for~\eqref{eq:CPlin} is described by
\[ u(t,\cdot) \sim e^{-t(-\Delta)^\theta/2}\,w(t,\cdot), \]
where~$w$ is the solution to the damping-free $\sigma$-evolution equation
\begin{equation}\label{eq:evolution}
w_{tt}+(-\Delta)^\sigma w=0, \qquad w(0,x)=0,\quad w_t(0,x)=u_1(x).
\end{equation}
The interplay between the diffusive part and the oscillating part of the solution now leads to a delicate equilibrium. Until now, it was not clear how to find suitable decay estimates to attack the critical exponent of problems~\eqref{eq:CPu} and~\eqref{eq:CPut}. Partial results for the existence were obtained in~\cite{DAR14}, but until now it was not clear if these results were close or far to optimal.

Indeed, the interplay of the diffusive part of the solution and the presence of oscillations leads to a decay estimate (see later, Remark~\ref{rem:loss}) which gives global existence of small data solutions to~\eqref{eq:CPu} for~$\alpha>(\sigma+2\theta)/(n-\sigma)$ (see, for instance, \cite{DAR14}). On the other hand, nonexistence of global solutions to~\eqref{eq:CPu} and, respectively, \eqref{eq:CPut}, has been proved in~\cite{DAE17NA} in the interval
\[ 0<\alpha\leq \frac{2\sigma}{n-\sigma}\,, \quad \text{and, respectively,} \qquad 0<\alpha\leq \frac{\sigma}n. \]
These nonexistence exponents are obtained by employing the test function method, and so are related to the scaling properties of the equations in~\eqref{eq:CPu} and~\eqref{eq:CPut}. However, in the employment of the scaling argument, the influence of the damping term~$(-\Delta)^\theta u_t$ disappears and, indeed, the parameter~$\theta$ does not appear above.

Until now, several conjectures have been made, in particular looking if the critical exponent~$\bar\alpha$ for~\eqref{eq:CPu} was somewhere between the values already obtained for existence and nonexistence, that is, if it was in the range
\[ \left(\frac{2\sigma}{n-\sigma}, \frac{\sigma+2\theta}{n-\sigma}\right). \]
A similar question arose for problem~\eqref{eq:CPut}.

\subsection*{The result, in brief}

In this paper, we give a final and someway apparently surprising answer to this question. The critical exponents for~\eqref{eq:CPu} and~\eqref{eq:CPut} in the noneffective case~$\theta\in(\sigma/2,\sigma]$ do not depend at all on~$\theta$, at least in low space dimension, and they are the ones obtained by scaling arguments and test function method in~\cite{DAE17NA}. In particular, they are the same obtained in the limit case~$\theta=\sigma/2$ (but in this latter, easier, case, they are valid in any space dimension~$n\geq1$).

We show how to get optimal $L^p-L^q$ decay estimates for the solution to~\eqref{eq:CPlin} taking advantage of both the diffusive and the oscillating part of the solution to~\eqref{eq:CPlin}, and how to properly use these estimates to prove global existence of small data solutions (in low space dimension) in the whole supercritical range suggested by the scaling properties of the equation.

It is important to remark that, even if the noneffective damping does not influence the critical exponent, the damping has a great influence on the regularity of the solution to our problems. It produces a smoothing effect which smooths out oscillations at high frequencies, in particular, allowing us to derive $L^1-L^1$ estimates, for instance, as done in~\cite{DAGL}. Such a property is typical of noneffective damping; in the effective case, the regularity of the solution is not influenced by the damping action.

The novel idea in this paper consists in treating separately the two components of the solution to~\eqref{eq:CPlin}. This strategy allows us to treat an equation which is not scale-invariant by splitting it into two terms with different scaling properties. In this way, we obtain optimal $L^p-L^q$ decay estimates to the nonlinear problems~\eqref{eq:CPu} and~\eqref{eq:CPut}. Our results are valid in low space dimension, leaving open the question if either the result remains valid in high space dimension, using a different proof, or the critical exponent changes in higher space dimension.

\subsection*{Plan of the paper}

The plan of the paper is the following:
\begin{itemize}
\item in Section~\ref{sec:results}, we collect and discuss our main results;
\item in Section~\ref{sec:localization} we localize the solution to~\eqref{eq:CPlin} at low and high frequencies;
\item in Section~\ref{sec:low}, we present our core result, obtaining low-frequencies $L^p-L^q$ estimates for the solution to~\eqref{eq:CPlin};
\item in Section~\ref{sec:lowder}, we show how to extend the results in Section~\ref{sec:low} to derive low-frequencies $L^p-L^q$ estimates for the derivatives of the solution to~\eqref{eq:CPlin};
\item in Section~\ref{sec:whatif}, we briefly discuss which estimates we may obtain if we do not employ the technique presented in our paper of splitting the kernels of the solution to~\eqref{eq:CPlin};
\item in Section~\ref{sec:loss}, we derive low-frequencies $L^p-L^q$ estimates with a loss of decay rate, out of the optimal range in Theorems~\ref{thlinearestimates} and~\ref{thm:lowder};
\item in Section~\ref{sec:high}, we derive high frequencies estimates for the solution to~\eqref{eq:CPlin} and its derivatives;
\item in Section~\ref{nonlinear}, we apply the decay estimates previously derived to prove Theorems~\ref{thnonlinear} and~\ref{thnonlinear2} for the nonlinear problems~\eqref{eq:CPu} and~\eqref{eq:CPut};
\item in~\ref{sec:Appendix}, we collect some multiplier theorems used to prove our estimates through the paper.
\end{itemize}

\subsection*{Notation used trough the paper}

In this paper, we use the following notation.
\begin{notation}
We write $ f\lesssim g$ when there exists a constant $C>0$ such that $f\leq C g$, and $f\approx g$ when $g\lesssim f \lesssim g $.

On the other hand, we write $f\sim g$ when the asymptotic profile of~$f$ is described by~$g$, in an appropriate sense (for instance, a pointwise estimate as~$\xii\to0$ or an estimate in a functional space as~$t\to\infty$).
\end{notation}
\begin{notation}
By~$\mathcal C_c^k=\mathcal C_c^k(\R^n)$, $k\in\N$, we denote the space of compactly supported, $k$-times differentiable functions with continuous derivatives. By~$\mathcal C_0^k=\mathcal C_0^k(\R^n)$, $k\in\N$, we denote the space of $k$-times differentiable functions with continuous derivatives, which vanish as~$|x|\to\infty$. By~$\mathcal S$, we denote the Schwartz space of functions with infinitely many rapidly decreasing derivatives, and by~$\mathcal S'$ we denote the space of tempered distributions, i.e. of the continuous linear functionals mapping~$\mathcal S$, equipped with its standard convergence, into~$\mathbb{C}$.
\end{notation}
\begin{notation}
We denote by~$\hat f=\mathfrak{F} f$ or~$\hat f(t,\cdot)=\mathfrak{F} f(t,\cdot)$ the Fourier transform, with respect to the space variable~$x$, of a tempered distribution or of a function, in the appropriate distributional or functional sense. We denote the inverse Fourier transform by~$\mathfrak{F}^{-1}$, in the appropriate sense.
\end{notation}
\begin{notation}
By~$L^p=L^p(\R^n)$, $p\in[1,\infty]$, we denote the space of measurable functions~$f$ such that~$|f|^p$ has finite integral over~$\R^n$, if~$p\in[1,\infty)$, or has finite essential supremum over~$\R^n$ if~$p=\infty$. We denote by~$W^{m,p}$, $m\in\N$, the space of~$L^p$ functions with weak derivatives up to the~$m$-th order in~$L^p$. We denote by~$H^s$, $s\geq0$, the space of~$L^2$ functions with~$(1+\xii^2)^{\frac{s}2}\,\hat u \in L^2$.
\end{notation}
\begin{notation}\label{DefLpqspaces}
By $L_p^q=L_p^q(\mathbb{R}^n)$ we denote the space of tempered distributions $T\in \mathcal S'$ such that~$T\ast f\in L^q$ for any~$f\in \mathcal S$, and
\[ \| T \ast f\|_{L^q} \leq C \|f\|_{L^p} \]
for all $f\in\mathcal S$ with a constant $C$, which is independent of $f$. In this case, the operator~$T\ast$ is extended by density from~$\mathcal S$ to~$L^p$.

By $M_p^q=M_p^q(\mathbb{R}^n)$, we denote the set of Fourier transforms $\hat{T}$ of distributions $T \in L_p^q$, equipped with the norm 
\[ \|m\|_{M_p^q}:=\sup\big\{\|\mathfrak{F}^{-1}(m\mathfrak{F}(f))\|_{L^q}:f\in \mathcal{S}, \|f\|_{L^p}=1\big\}.\]
and we set~$M_p=M_p^p$. A function~$m$ in $M_p^q$ is called a multiplier of type $(p,q)$. 
\end{notation}
In this paper we will also make us of a dyadic partition of unity and of the related notion of Besov space (see~\cite{Tri83}).
\begin{notation}\label{not:Besov}
We fix a nonnegative function $\psi\in\mathcal C^\infty$, having compact support in $\{ \xi \in \mathbb{R}^n : 2^{-1}\leq \xii\leq 2\}$, such that:
\begin{equation}\label{eq:partition}
\sum_{k=-\infty}^{+\infty} \psi_k(\xi)=1, \qquad \text{where~$\psi_k(\xi):=\psi(2^{-k} \xi)$.}
\end{equation}
(This property is easily obtained if~$\psi(\xi)=\varphi(\xi/2)-\varphi(\xi)$, for some~$\varphi\in\mathcal C^\infty$, with~$\varphi(\xi)=1$ for~$\xii\leq1/2$ and~$\varphi(\xi)=0$ if~$\xii\geq1$). For any~$p\in[1,\infty]$, we define the Besov space
\[ B^0_{p,2} = \{ f\in\mathcal S': \ \forall k\in\Z, \ \mathfrak{F}^{-1}(\psi_k\hat f)\in L^p, \quad \|f\|_{B^0_{p,2}}<\infty \}, \]
where
\[ \|f\|_{B^0_{p,2}} = \|\mathfrak{F}^{-1}(\psi_k\hat f)\|_{\ell^2(L^p)} = \left(\sum_{k=-\infty}^{+\infty} \|\mathfrak{F}^{-1}(\psi_k\hat f)\|_{L^p}^2\right)^{\frac12}. \]
\end{notation}

\section{Results}\label{sec:results}

We are now ready to state our main results.
\begin{theorem}\label{thnonlinear}
Assume that~$\sigma>1$ and that the damping is \emph{noneffective}, i.e., $2\theta\in(\sigma,2\sigma]$. Also assume that the space dimension~$n$ verifies $\sigma<n\leq\bar n(\sigma)$, where
\begin{equation}\label{eq:barn}
\bar n(\sigma) = (3\sigma-2)\left[1 + \frac12\Big(\sqrt{1+8\sigma(3\sigma-2)^{-2}}-1\Big)\right].
\end{equation}
Fix~$\alpha>\alpha_0$, where
\begin{equation}\label{eq:alpha0}
\alpha_0=\frac{2\sigma}{n-\sigma}\,.
\end{equation}
Then there exists a constant $\epsilon>0$ such that for any
\begin{equation}\label{eq:datau}
u_1\in L^{1}\cap L^{\eta} \quad\text{with}\qquad \|u_1\|_{L^1}+\|u_1\|_{L^\eta} < \epsilon,\qquad \text{where~$\eta=\max\{2,n/(2\theta)\}$,}
\end{equation}
there exists a uniquely determined energy solution $u\in \mathcal C([0, \infty),  H^{\sigma}\cap L^\infty) \cap \mathcal C^1([0, \infty), L^{2})$ to~\eqref{eq:CPu}.
Moreover, the solution satisfies the energy estimate
\begin{equation*}
E(t)=\frac12 \|u_t(t,\cdot)\|_{L^2}^2+\frac12\|(-\Delta)^{\frac\sigma2}u(t,\cdot)\|_{L^2}^2\leq C\,(1+t)^{-\frac{n}{2\theta}}\,\big(\|u_1\|_{L^1}^2+\|u_1\|_{L^\eta}^2\big),
\end{equation*}
the decay estimates
\[ \|u(t,\cdot)\|_{L^q}\leq C\,(1+t)^{1-\frac{n}\sigma\left(1-\frac1q\right)}\,\big(\|u_1\|_{L^1}+\|u_1\|_{L^\eta}\big), \qquad \forall q\in[1+\alpha,\infty],\]
and the estimate
\begin{equation}\label{eq:uL2} \|u(t,\cdot)\|_{L^2}\leq \begin{cases}
C\,(1+t)^{1-\frac{n}{2\sigma}}\,\big(\|u_1\|_{L^1}+\|u_1\|_{L^\eta}\big), & \text{if~$n<2\sigma$,}\\
C\,\log(e+t)\,\big(\|u_1\|_{L^1}+\|u_1\|_{L^\eta}\big), & \text{if~$n=2\sigma$,}\\
C\,(1+t)^{-\frac{n-2\sigma}{4\theta}}\,\big(\|u_1\|_{L^1}+\|u_1\|_{L^\eta}\big), & \text{if~$n>2\sigma$.}
\end{cases}
\end{equation}
The constant~$C>0$ does not depend on the initial data.
\end{theorem}
The optimality of estimate~\eqref{eq:uL2} in Theorem~\ref{thnonlinear} is guaranteed for~$\theta=\sigma$ by Theorems 1.1, 1.2 and 1.3, in~\cite{Ikenew} in the cases $n>2\sigma$, $n=2\sigma$ and $n<2\sigma$, respectively; in particular, we see that some log-loss cannot be avoided when~$n=2\sigma$.
\begin{remark}\label{rem:barn}
We notice that $\bar n(\sigma)$ in~\eqref{eq:barn} is the solution to the second order equation
\begin{equation}\label{eq:barneq}
n^2 - (3\sigma-2) n - 2\sigma =0,
\end{equation}
and it satisfies the following properties: $\bar n(\sigma)\sim 3\sigma-2$ as~$\sigma\to\infty$, and~$\bar n(\sigma)-(3\sigma-2)$ is a decreasing function with respect to~$\sigma$, with its infimum given by~$\bar n(\sigma)-(3\sigma-2)\to 1$ as~$\sigma\to1$. In particular, $\bar n(\sigma)\in(3\sigma-2,3\sigma-1)$ for any~$\sigma>1$. Equation~\eqref{eq:barneq} corresponds to~\eqref{eq:rangebotheq} with~$p=1$ and~$q=1+\alpha_0$, i.e. to the maximum range for the space dimension~$n$, in which we may apply Theorem~\ref{thlinearestimates} with~$p=1$ and~$q=1+\alpha_0$.

Theorem~\ref{thnonlinear} remains valid, indeed, for~$\sigma\in(0,1)$, as well, due to the fact that the Theorem~\ref{thlinearestimates} is also valid for~$\sigma\in(0,1)$ (while our proof is not valid for~$\sigma=1$, see later, \eqref{eq:Hessian}). However, due to the expression of~$\bar n(\sigma)$ in~\eqref{eq:barn}, it only provides a result in space dimension~$n=1$ for~$\sigma\in[2/5,1)$.
\end{remark}
In particular, Theorem~\ref{thnonlinear} applies to the case of plate equation, $\sigma=2$.
\begin{example}
Let~$n=3,4$, $\theta\in(1,2]$, and consider the semilinear damped plate equation
\begin{equation}
\label{eq:CPplate}
\begin{cases}
u_{tt}+\Delta^2 u +(-\Delta)^\theta u_t=|u|^{1+\alpha}, \quad x\in \R^n, \, t\in\R_+,\\
u(0,x)=0, \\
u_t(0,x)=u_1(x).
\end{cases}
\end{equation}
Fix~$\alpha>4$ if~$n=3$ and~$\alpha>2$ if~$n=4$. Then there exists a constant $\epsilon>0$ such that for any~$u_1$ as in~\eqref{eq:datau}, there exists a uniquely determined energy solution $u\in \mathcal C([0, \infty),  H^2) \cap \mathcal C^1([0, \infty), L^{2})$ to  \eqref{eq:CPplate}. Moreover, the solution satisfies the energy estimate
\begin{equation*}
E(t)=\frac12 \|u_t(t,\cdot)\|_{L^2}^2+\frac12\|\Delta u(t,\cdot)\|_{L^2}^2\leq C\,(1+t)^{-\frac{n}{2\theta}}\,\big(\|u_1\|_{L^1}^2+\|u_1\|_{L^\eta}^2\big),
\end{equation*}
and the decay estimates
\[ \|u(t,\cdot)\|_{L^q}\leq C\,(1+t)^{1-\frac{n}2\left(1-\frac1q\right)}\,\big(\|u_1\|_{L^1}+\|u_1\|_{L^\eta}\big), \qquad \forall q\in[1+\alpha,\infty].\]
We mention that some plate models include also a term~$-\Delta u_{tt}$ called rotational inertia. Linear estimates for these models, for which a regularity-loss type decay appears, have been investigated in~\cite{CD09, CDI13p, CHI16, SK10}.
\end{example}
If~$\sigma\geq3$, we may also derive an existence result for problem~\eqref{eq:CPut}.
\begin{theorem}\label{thnonlinear2}
Assume that~$\sigma\geq3$ and that the damping is \emph{noneffective}, i.e., $2\theta\in(\sigma,2\sigma]$. Also assume that~$n\leq\sigma-2$. Fix~$\alpha>\alpha_1$, where
\begin{equation}\label{eq:alpha1}
\alpha_1=\frac{\sigma}n\,.
\end{equation}
Then there exists a constant $\epsilon>0$ such that for any
\begin{equation}\label{eq:dataut}
u_1\in L^{1}\cap L^{1+\alpha} \quad\text{with}\qquad \|u_1\|_{L^1}+\|u_1\|_{L^{1+\alpha}} < \epsilon,
\end{equation}
there exists a uniquely determined energy solution $u\in \mathcal C([0, \infty),H^{\sigma}) \cap \mathcal C^1([0, \infty), L^2\cap L^{1+\alpha})$ to~\eqref{eq:CPut}. 

Moreover, the solution satisfies the energy estimate
\begin{equation*}
E(t)=\frac12 \|u_t(t,\cdot)\|_{L^2}^2+\frac12\|(-\Delta)^{\frac\sigma2}u(t,\cdot)\|_{L^2}^2\leq C\,(1+t)^{-\frac{n}{2\theta}}\,\big(\|u_1\|_{L^1}^2+\|u_1\|_{L^{1+\alpha}}^2\big),
\end{equation*}
and the decay estimates
\begin{align*}
\|u(t,\cdot)\|_{L^\infty}
    & \leq C\,(1+t)^{1-\frac{n}\sigma}\,\big(\|u_1\|_{L^1}+\|u_1\|_{L^{1+\alpha}}\big),\\
\|u_t(t,\cdot)\|_{L^{1+\alpha}}
    & \leq C\,(1+t)^{-\frac{n}\sigma\left(1-\frac1{1+\alpha}\right)}\,\big(\|u_1\|_{L^1}+\|u_1\|_{L^{1+\alpha}}\big).
\end{align*}
The constant~$C>0$ does not depend on the initial data.
\end{theorem}
The nonexistence counterpart of Theorems~\ref{thnonlinear} and~\ref{thnonlinear2} has been given in~\cite{DAE17NA} for integer powers~$\theta$ and~$\sigma$, in both the effective and noneffective cases, using a test function method which goes back to~\cite{MP} and some strategies introduced in~\cite{DAL03, MP09}. More in general, by using a novel test function recently developed in~\cite{F}, the nonexistence result remains valid in the subcritical ranges for fractional powers~$\theta$ and~$\sigma$, see Examples 6.3 and 6.4 in~\cite{DAF20}. Summarizing, we have the following.
\begin{proposition}
Let~$0\leq \theta\leq\sigma$, and assume that~$u_1\in L^1$ verifies
\begin{equation}\label{eq:datatest}
\int_{\R^n} u_1(x)\,dx >0.
\end{equation}
Then there exists no global (weak) solution to~\eqref{eq:CPu}:
\begin{itemize}
\item for any~$\alpha>0$ if~$n\leq \min\{2\theta,\sigma\}$;
\item for any
\[ \alpha\in\left(0, \frac{2\sigma}{n-\min\{2\theta,\sigma\}}\right),\]
if~$n>\min\{2\theta,\sigma\}$.
\end{itemize}
Moreover, there exists no global weak solution to~\eqref{eq:CPut} for any
\[ \alpha\in\left(0, \frac{\min\{2\theta,\sigma\}}n\right).\]
The nonexistence in the critical cases remains valid if both~$\sigma$ and~$\theta$ are integers (and in some other cases, see~\cite{DAF20}).
\end{proposition}
The proof of Theorems~\ref{thnonlinear} and~\ref{thnonlinear2} is heavily based on the possibility to obtain optimal $L^p-L^q$ decay estimates, $1\leq p\leq q\leq\infty$, by exploiting the oscillating part of the fundamental solution and ignoring its diffusive part, which would produce a worse decay, due to the different scaling properties. However, this optimality is valid in a range~$(p,q)$ depending on the space dimension~$n$, described by condition~\eqref{eq:rangeboth} below, which is related to the evolution part of the equation. Indeed, this condition is consistent with the one considered in~\cite{EL} for the damping-free $\sigma$-evolution equation in~\eqref{eq:evolution}.

In particular, for powers~$\alpha$ close to the critical exponent~$\alpha_0$ and~$\alpha_1$, condition~\eqref{eq:rangeboth} is valid only in low space dimension for the $L^1-L^{1+\alpha}$ estimate. This restriction allows us to use this technique to derive a sharp global existence result for problems~\eqref{eq:CPu} and~\eqref{eq:CPut} only in low space dimension. The problem to find the critical exponent in high space dimension remains open.
\begin{theorem}\label{thlinearestimates}
Let~$\sigma>1$ and assume a noneffective damping, that is, $\theta\in(\sigma/2,\sigma]$. Let~$1\leq p \leq r \leq q \leq \infty$, be such that
\begin{equation}\label{eq:rangeboth}
\frac{n}\sigma\left(\frac1p-\frac1q\right) + n\max\left\{\left(\frac12-\frac1p\right),\left(\frac1q-\frac12\right)\right\}<1.
\end{equation}
and
\begin{equation}\label{eq:regLr}
n\left(\frac1r-\frac1q\right)\leq 2\theta,
\end{equation}
if~$(r,q)\in(1,\infty)$, or~$n(1/r-1/q)<2\theta$ if~$r=1$ or~$q=\infty$.

If  $u_1\in L^p\cap L^r$,  then the solution $u$ to the Cauchy problem \eqref{eq:CPlin} satisfies the following estimate
\begin{equation}\label{linearenergyestimatesth}
\| u(t,\cdot)\|_{L^q} \lesssim\, (1+t)^{1-\frac{n}{\sigma}\left( \frac1{p}- \frac1{q}\right)}\,\|u_1\|_{L^p} + e^{-ct}\,\|u_1\|_{L^r}, \qquad  \forall t\geq 0.
\end{equation}
In particular, if we take~$r=p$ in assumption~\eqref{eq:regLr}, then we get the $L^p-L^q$ estimate
\begin{equation}\label{eq:LpLq}
\| u(t,\cdot)\|_{L^q} \lesssim\, (1+t)^{1-\frac{n}{\sigma}\left( \frac1{p}- \frac1{q}\right)}\,\|u_1\|_{L^p}, \qquad  \forall t\geq 0.
\end{equation}
Moreover, if equality holds in~\eqref{eq:rangeboth}, that is,
\begin{equation}\label{eq:rangebotheq}
\frac{n}\sigma\left(\frac1p-\frac1q\right) + n\max\left\{\left(\frac12-\frac1p\right),\left(\frac1q-\frac12\right)\right\}=1,
\end{equation}
then estimate~\eqref{linearenergyestimatesth} remains valid with a possible log-loss, that is,
\begin{equation}\label{eq:loglimit}
\| u(t,\cdot)\|_{L^q} \lesssim\, (1+t)^{n\max\left\{\left( \frac12-\frac1{p}\right), \, \left( \frac1q-\frac12\right)\right\}}\,\log(e+t)\,\|u_1\|_{L^p} + e^{-ct}\,\|u_1\|_{L^r}, \qquad  \forall t\geq 0;
\end{equation}
if~$1<p\leq 2\leq q<\infty$, the log-loss may be avoided and we obtain~\eqref{linearenergyestimatesth}, namely,
\begin{equation}\label{eq:nologloss}
\| u(t,\cdot)\|_{L^q} \lesssim\, (1+t)^{n\max\left\{\left( \frac12-\frac1{p}\right), \, \left( \frac1q-\frac12\right)\right\}}\,\|u_1\|_{L^p} + e^{-ct}\,\|u_1\|_{L^r}, \qquad  \forall t\geq 0.
\end{equation}
\end{theorem}
A more general version of Theorem~\ref{thlinearestimates}, which includes derivatives of the solutions, is obtained combining Theorem~\ref{thm:lowder} and Theorem~\ref{thm:sing}, which provide, respectively, $L^p-L^q$ low frequencies estimates and~$L^r-L^q$ high frequencies estimates. Out of the range~$(p,q)$ determined by~\eqref{eq:rangeboth}, it is still possible to derive suitable $L^p-L^q$ decay estimates partially taking advantage of the diffusive part of the solution, but exploiting at most its oscillating part, as we do in Theorem~\ref{thm:loss}. Still, to treat the two components of the solution separately, a condition appears which restricts the~$(p,q)$ range.

The novel idea in this paper consists in treating separately the two components of the solution, the oscillating one, and the diffusive one. However, in some cases, this is not possible, as it happens in high space dimension, and mixing the two components together becomes necessary. For instance, $L^1-L^1$ estimates in high space dimension for the model with~$\sigma=\theta$ have been recently derived in~\cite{DAGL}. In the case studied therein, the strategy to split the two components of the solution was not possible.
\begin{remark}
The use of different regularities~$L^p$ and~$L^r$ in Theorem~\ref{thlinearestimates} is related to the different behavior of the solution operator at low and high frequencies, where we use~$L^p$ and, respectively, $L^r$ regularity of the data. Indeed, for a fixed~$q$, taking smaller values of~$p$ w.r.t~$q$ produces a higher decay rate in~\eqref{linearenergyestimatesth}, but managing smaller values of~$r$ w.r.t.~$q$ becomes more difficult, due to~\eqref{eq:regLr}. In Section~\ref{sec:high} we see how condition~\eqref{eq:regLr} may be removed (or relaxed when~$\sigma=\theta$) due to the \emph{smoothing effect}, if we allow a singularity at~$t=0$ (indeed a singularity~$t^{-\delta}$ may be managed when dealing with the nonlinear problems, if~$\delta\in(0,1)$).
\end{remark}
\begin{remark}
Let~$q=p'=p/(p-1)$, the H\"older conjugate of~$p$, that is, $1/p+1/q=1$. Then condition~\eqref{eq:rangeboth} is verified for any~$p\in[1,2]$ if~$\sigma>2n/(n+2)$, and for any~$p\in(1,2]$ such that
\begin{equation}\label{eq:conjugateok}
n\left(\frac1p-\frac12\right)<\frac\sigma{2-\sigma},
\end{equation}
otherwise. On the other hand, if~$p=q$, then condition~\eqref{eq:rangeboth} is verified if
\begin{equation}\label{eq:regularity}
n\left|\frac1p-\frac12\right| <1.
\end{equation}
This latter (together with the case in which the equality holds when~$p\neq1,\infty$), is the sharp condition to get a $L^p-L^p$ estimate for the $\sigma$-evolution equation damping-free. We remark that condition~\eqref{eq:conjugateok} is less restrictive than condition~\eqref{eq:regularity} for any~$\sigma>1$, but the limit of condition~\eqref{eq:conjugateok} as~$\sigma\to1$ gives~\eqref{eq:regularity}. 
\end{remark}
\begin{remark}
The control from above of the term
\[ \frac{n}\sigma\left(\frac1p-\frac1q\right) \]
in condition~\eqref{eq:rangeboth} is related to the decay rate profile for the solution to~\eqref{eq:CPlin} which originates from taking different spaces for the solution and the initial data. Larger distances, producing larger decay, are more difficult to control. On the other hand, the term
\[ n\max\left\{\left(\frac12-\frac1p\right),\left(\frac1q-\frac12\right)\right\} \]
may have a positive or negative sign. The sign is negative, when~$p\leq 2\leq q$, whereas the sign is positive when~$p\leq q\leq 2$ or, respectively, $2\leq p\leq q$. More precisely, it is~$n(1/q-1/2)$ or, respectively, $n(1/2-1/p)$. This term represents how more difficult become to control even $L^q-L^q$ or, respectively $L^p-L^p$ estimates when one goes away from the line~$p=q=2$, as it happens in evolution equations damping-free. When~$p\leq2\leq q$, this difficulty does not appear, since one rely on different methods, as Hausdorff-Young inequality or stationary phase methods, to derive the desired decay estimates, completely avoiding the theory of multipliers on~$L^p$. Indeed, it is well known that obtaining $L^p-L^q$ estimates is much easier when~$p\leq 2\leq q$.
\end{remark}

\subsection*{The difficulties arising in the limit case~$\sigma=1$}

As mentioned in Remark~\ref{rem:barn}, one may easily verify that Theorem~\ref{thlinearestimates} remains valid for $\sigma \in (0,1)$, as well as Theorems~\ref{thm:lowder} and~\ref{thm:sing}. However, this case has a limited interest, so we prefer to assume~$\sigma>1$, for brevity, to emphasize that our proof fails in the threshold case~$\sigma=1$, due to the lack of~\eqref{eq:Hessian}.

\bigskip

We expect that the case~$\sigma=1$ may still be treated with our approach, deriving a result similar to the one obtained for~$\sigma\neq1$, but with some influence from the fact that the Hessian of the function~$\xii$ is singular. For the wave equation with viscoelastic damping ($\sigma=\theta=1$ in~\eqref{eq:CPlin}), $L^1-L^{\infty}$ low-frequencies estimates for the solutions are obtained using the stationary phase method in \cite{S00}, and the decay rate~$(1+t)^{-\frac{3(n-1)}4}$ is derived in space dimension~$n\geq2$.

Let~$n\geq2$ and~$\theta=\sigma$. In space dimension~$n=2$, we may apply Theorem~\ref{thlinearestimates} for any~$\sigma>1$, and we get the decay rate~$(1+t)^{1-\frac2\sigma}$, which tends to~$(1+t)^{-1}$ as~$\sigma\to1$. This decay rate is better than the decay~$(1+t)^{-\frac34}$ derived in~\cite{S00} when~$\sigma=1$. On the other hand, let~$n\geq3$ and assume that~$\sigma\in(1,2n/(n+2))$. Applying Theorem~\ref{thm:loss}, we get the decay rate~$(1+t)^{-\frac{n-2}4-\frac{n}{2\sigma}}\,\log(e+t)$. As~$\sigma\to1$, this decay rate tends to~$(1+t)^{-\frac{3n-2}4}\,\log(e+t)$, which is again better than the decay rate in~\cite{S00} when~$\sigma=1$. It remains an open problem to show that the decay rate power is discontinuous at~$\sigma=1$ (likely, as a consequence of the singularity of the Hessian matrix).

\bigskip

The probable loss of decay rate appearing in the special case~$\sigma=1$ has the consequence that we cannot obtain in this case the same critical exponent~$\alpha_0=2\sigma/(n-\sigma)$, as in Theorem~\ref{thnonlinear}. In particular, let~$n=2$ and~$\sigma\in(1,2)$, with~$2\theta\in(\sigma,2\sigma]$. Theorem~\ref{thnonlinear} guarantees the global existence of small data energy solutions to~\eqref{eq:CPu} for~$\alpha>2\sigma/(2-\sigma)$. As~$\sigma\to1$, this value tends to~$2$. However, we do not expect that the critical exponent is~$2$ when~$\sigma=1$ and~$\theta\in(1/2,1]$ (by the results in~\cite{DAR14}, we know that global solutions exist for~$\alpha>1+2\theta$, but this result is very likely not optimal; see also the $L^1-L^1$ estimates obtained in~\cite{NR13}).

\bigskip

A hint on the situation of the effectively damped wave equation may come from the case of the undamped wave. The results obtained in \cite{P, S} imply that the solutions to the Cauchy problem for the free wave equation
\[u_{tt}-\Delta u =0, \quad
u(0,x)=0, \quad
u_t(0,x)=u_1(x),
\]
satisfies the $L^p-L^q$ estimates
\[\| u(t,\cdot)\|_{L^q} \lesssim\, (1+t)^{1-n\left( \frac1{p}- \frac1{q}\right)}\,\|u_1\|_{L^p}
\]
if, and only if,
\begin{equation}\label{sigma1}
n\left(\frac1p-\frac1q\right) + (n-1)\max\left\{\left(\frac12-\frac1p\right),\left(\frac1q-\frac12\right)\right\}\leq1.
\end{equation}
Unfortunately, these estimates are not of interest to treat the power nonlinearity~$|u|^{1+\alpha}$, due to the fact that condition \eqref{sigma1} is not satisfied for the pair $(p, q)=(1, 1+\alpha_0)$ for any~$n\geq2$, so we can not follow the ideas of the proof of Theorem \ref{thnonlinear} to  derive the critical exponent $\alpha_0=2/(n-1)$ for the free wave equation. On the other hand, for the case~$\sigma\in(1,2)$ in space dimension~$n=2$ the critical exponent for the undamped $\sigma$-evolution equation is still~$\alpha_0=2\sigma/(n-\sigma)$, see~\cite{EL}.

Indeed, it is well known that  the critical exponent for the undamped wave equation with power nonlinearity~$|u|^{1+\alpha}$ is the exponent~$\alpha_{\mathrm{Strauss}}$ conjectured by W.A. Strauss~\cite{Strauss} (see also~\cite{Georgiev, G, G2, GLS, Sc, Si}), which solves the algebraic equation
\[ \frac{n-1}2\,\alpha(\alpha+1)=\alpha+2. \]
This latter is strictly bigger than~$2/(n-1)$. For this reason, we expect that the critical exponent for~\eqref{eq:CPu} is ``somewhere'' between $2/(n-1)$ and~$\alpha_{\mathrm{Strauss}}$, when~$\sigma=1$. Possibly, it tends to~$2/(n-1)$ when~$\theta\to1/2$.

\section{Localization of the solution at low and high frequencies}\label{sec:localization}

We denote by
\[ \hat u(t,\xi) = \mathscr{F} [u(t,\cdot)](\xi) = \int_{\R^n} e^{-ix\xi}\,u(t,x)\,dx \]
the Fourier transform of~$u(t,x)$ with respect to the space variable. Then~$\hat u$ solves the Cauchy problem for the damped harmonic oscillator
\begin{equation}
\label{CPF}
\begin{cases}
\hat u_{tt}+\xii^{2\sigma} \hat u +\xii^{2\theta} \hat u_t=0, \quad t\in\R_+,\\
\hat u(0,\xi)=0, \\
\hat u_t(0,\xi)=\hat u_1(\xi),
\end{cases}
\end{equation}
for any~$\xi\in\R^n$. If we write
\[ u(t,x) = K(t,x)\ast_{(x)} u_1(x), \]
where~$K$ is the fundamental solution to~\eqref{CPF}, then
\[ \hat K(t,\xi) = t\,e^{-t\xii^{2\theta}/2}\,\sinc (t\omega), \quad \omega=\xii^\sigma\,\sqrt{1-\xii^{4\theta-2\sigma}/4}, \]
for any~$\xi$ such that~$\xii^{2\theta-\sigma}<2$, whereas
\[ \hat K(t,\xi) = \frac{e^{\lambda_+t}-e^{\lambda_-t}}{\lambda_+-\lambda_-}\,, \quad \lambda_\pm(\xi) = -\frac12\,\xii^{2\theta} \big(1\mp \sqrt{1-4\xii^{2\sigma-4\theta}}\big). \]
In particular, $\lambda_-(\xi)\sim -\xii^{2\theta}$ and~$\lambda_+(\xi)\sim-\xii^{\sigma-\theta}$ as~$\xii\to\infty$.

It is clear that for any~$t\geq0$, $\hat K(t,\cdot)$ is smooth in~$\R^n\setminus\{0\}$. To deal with~$K$, it is convenient to localize it at low and high frequencies. We fix small~$\eps_0\in(0,1)$ and large~$N_\infty\gg2^{\frac{1}{2\theta-\sigma}}$, and we fix~$\varphi_{0}\in\mathcal C_c^\infty(\R^n)$,  and $\varphi_\infty$ in $\mathcal C^{\infty}(\mathbb{R}^n)$ such that
\begin{equation}\label{eq:var}
\varphi_0(\xi)=\begin{cases}
                     1 & \text{if~$\abs{\xi}\leq \eps_0/2$,} \\
                     0 & \text{if~$\abs{\xi}\geq \eps_0$,}
                   \end{cases} \qquad
\varphi_\infty(\xi)=\begin{cases}
                     1 & \text{if~$\abs{\xi}\geq 2N_\infty$,} \\
                     0 & \text{if~$\abs{\xi}\leq N_\infty$.}
                   \end{cases}
\end{equation}
We also put $\varphi_1=1-(\varphi_0+\varphi_\infty)\in \mathcal C_c^{\infty}(\mathbb{R}^n)$. We now define
\[ K_j= \mathfrak{F}^{-1} (\varphi_j\,\hat K), \]
so that~$K=K_0+K_1+K_\infty$, where~$K_0$, $K_1$ and~$K_\infty$ are the localization of the fundamental solution at low, intermediate, and high frequencies.

Let~$\beta\in\N^n$, $\ell\in\N$ and~$b\geq0$. At intermediate frequencies, the estimate
\[ \|\partial_x^\beta (-\Delta)^{\frac{b}2} \partial_t^\ell K_1(t, \cdot)\ast u_1\|_{L^q} \leq C\,e^{-ct}\,\|u_1\|_{L^p}, \]
trivially follows for any~$1\leq p\leq q\leq\infty$, for any~$t\geq0$, for some~$C,c>0$, independent of the data. Indeed, the claim follows from the fact that
\[ \|\partial_x^\beta (-\Delta)^{\frac{b}2}\partial_t^j K_1(t, \cdot)\|_{L^1} \leq \| (i\xi)^\beta \xii^b \varphi_1\,\partial_t^j \hat K(t, \cdot)\|_{H^m} \leq Ce^{-ct}, \]
for~$m>n/2$, integer, and
\[ \|\partial_x^\beta  (-\Delta)^{\frac{b}2} \partial_t^j K_1(t, \cdot)\|_{L^\infty} \leq \| (i\xi)^\beta\xii^b \varphi_1\,\partial_t^j \hat K(t, \cdot)\|_{L^1} \leq Ce^{-ct}, \]
for some~$C,c>0$, so that it is sufficient to apply Young inequality. For this reason, as it is expected for this kind of problems, we may focus our attention in the study at low and high frequencies. Our main interest is into derive new estimates at low frequencies, where oscillations appear, based on a new strategy to approach the analysis of~$K_0(t,x)$.

\section{$L^p-L^q$ low frequencies estimates ($\xii\leq\eps_0$) for the solution}\label{sec:low}

First of all, we give a straight-forward regularity result.
\begin{proposition}\label{ppsmalltime}
Let $T\geq1$. Then, for any $t\in[0,T]$ and $1\leq p\leq q\leq \infty$, it holds
\[
\| K_0(t,\cdot)\ast u_1\|_{L^q}\leq C(T) \norm{u_1}_{L^p},
\]
for some~$C(T)$, independent of~$u_1$.
\end{proposition}
\begin{proof}
To prove Proposition~\ref{ppsmalltime}, it is sufficient to show that~$K_0(t,\cdot)\in L^1\cap L^\infty$ and apply Young inequality. In order to do that we apply Lemma~\ref{lem:parts}. We notice that~$|\omega^{-1}\sin(t\omega)|\leq t$. On the other hand,
\begin{equation}
\label{eq:Taylorsin}
\omega^{-1}\sin(t\omega) = t - \frac1{6}\,t^3\omega^2\,\int_0^1 (1-\rho)^3\sin(\rho\omega t)\,d\rho,
\end{equation}
by Taylor's formula, so that we easily derive
\[ \forall\gamma\neq0: \qquad |\partial_\xi^\gamma (\omega^{-1}\sin(t\omega))| \lesssim t^3\,\xii^{2\sigma-|\gamma|}, \]
for any~$\xii\leq \eps_0$. On the other hand,
\[ \forall\gamma\neq0: \qquad |\partial_\xi^\gamma e^{-t\xii^{2\theta}/2}| \lesssim t\,\xii^{2\theta-|\gamma|}. \]
Therefore, recalling that~$\theta\leq\sigma$ and that~$\varphi_0\in\mathcal C_c^\infty$, we get
\[ \forall \gamma: \qquad \big|\partial_\xi^\gamma \big(\varphi_0\hat K(t,\xi)\big)\big| \lesssim C(T)\,(1+\xii^{2\theta-|\gamma|}). \]
If~$2\theta>1$, we may now apply the first part of Lemma~\ref{lem:parts} with~$\kappa=n+1$, obtaining~$|K_0(t,x)|\leq C(T)(1+|x|)^{-n-1}$, and this concludes the proof. If~$2\theta\in(0,1]$, then we may apply the second part of Lemma~\ref{lem:parts} with~$\kappa=n$, $a=n-2\theta$ and~$a_1=a+1$, obtaining
\[ |K_0(t,x)|\leq \begin{cases}
C(T)(1+|x|)^{-n-2\theta} & \text{if~$2\theta\in(0,1)$,} \\
C(T)(1+|x|)^{-n-1}\,\log(e+|x|) & \text{if~$2\theta=1$,}
\end{cases} \]
and this concludes the proof.
\end{proof}
In view of Proposition~\ref{ppsmalltime}, with no loss of generality, in this section we may now assume~$t\geq1$.

For any $1\leq p\leq q\leq \infty$ we now estimate
\begin{equation}\label{split}
\|K_0(t,\cdot)\|_{L^q_p}\leq t\,\|\mathscr{F}^{-1}(e^{-t|\xi|^{2\theta}/2})\|_{L^1}\,\|\mathscr{F}^{-1}(\sinc(\omega t)\varphi_0)\|_{L^q_p}=Ct\,\|\mathscr{F}^{-1}(\sinc(\omega t)\varphi_0)\|_{L^q_p},
\end{equation}
where we used that
\[ \|\mathscr{F}^{-1}(e^{-t|\xi|^{2\theta}/2})\|_{L^1}= \|\mathscr{F}^{-1}(e^{-|\xi|^{2\theta}/2})\|_{L^1}=C,\]
for any~$t>0$.

We will now focus our attention on the oscillating part of the fundamental solution, forgetting about its diffusive part. As we will see later in Section~\ref{sec:loss}, this is the best strategy for $L^p-L^q$ estimates only in some $(p,q)$ range, the one that we are interested to prove the existence result in the whole supercritical range of powers~$\alpha$ in Theorems~\ref{thnonlinear} and~\ref{thnonlinear2}.

It is important to remark now that~$\varphi_0\sinc t\omega$ is not scale-invariant as it happens to~$\sinc t\xii^\sigma$ with a $\sigma$-evolution equation without damping as in~\eqref{eq:evolution}. Indeed, in such a case, one would be able to derive
\[ \| \mathscr{F}^{-1}(\sinc(\xii^\sigma t))\|_{L^q_p}= t^{-\frac{n}{\sigma}(\frac{1}{p}-\frac{1}{q})}\| \mathscr{F}^{-1}(\sinc(\xii^\sigma))\|_{L^q_p}, \]
and work directly with a time-independent kernel. However, in our case, even without the scale-invariance, we may still perform a change of variable which allows to treat the time-dependent part as a perturbation.

By the change of variable $\eta=t^{\frac{1}{\sigma}}\xi$, for any $1\leq p\leq q\leq \infty$, it holds
\begin{equation}
\| \mathscr{F}^{-1}(\sinc(\omega t)\varphi_0)\|_{L^q_p}= t^{-\frac{n}{\sigma}(\frac{1}{p}-\frac{1}{q})}\| \tilde K_0(t,\cdot)\|_{L^q_p},
\end{equation}
where
\begin{align*}
\tilde K_0(t,x)&= \mathscr{F}^{-1}(\sinc(\tilde\omega(t,\eta))\tilde\varphi_0(t,\eta)),\\
\tilde\omega(t,\eta)&=|\eta|^\sigma \,\sqrt{1-t^{2-\frac{4\theta}{\sigma}}|\eta|^{4\theta-2\sigma}/4},\\
\tilde\varphi_0(t,\eta)&=\varphi_0(t^{-\frac{1}{\sigma}}\eta).
\end{align*}
Clearly, $\tilde K_0(t,\cdot)$ is supported in~$\{|\eta|\leq \eps_0\,t^\frac{1}{\sigma}\}$ and $\tilde\omega(t,\eta)\approx |\eta|^\sigma$. More precisely,
\begin{equation}\label{eq:perturbation}
\tilde\omega(t,\eta) = |\eta|^\sigma + \textit{O}(\eps_0),
\end{equation}
together with its derivatives. The main reason to perform a change of variable is that, due to the oscillations, the derivatives of~$\sinc\tilde\omega$ have a different behavior for small~$|\eta|$ and for large~$|\eta|$. We emphasize that large values of~$|\eta|$ are possible when~$t\gg\eps_0^{-\sigma}$. Using that
\[ |\sinc^{(k)}\rho|\leq C\,(1+\rho)^{-1},\]
we obtain
\[
\forall\gamma\neq0: \quad |\partial_\eta^{\gamma} \sinc\tilde\omega|\lesssim \begin{cases}
|\eta|^{-\sigma+(\sigma-1)|\gamma|} &\text{ if } 1<|\eta|<\eps_0t^{\frac{1}{\sigma}},\\
|\eta|^{2\sigma-|\gamma|} &\text{ if } |\eta|<2. \end{cases}
\]
We now split our analysis in two cases, considering small and large values of the ``new frequencies'' $\eta$. We fix~$\chi\in\mathcal C_c^\infty$, supported in~$\{|\eta|\leq2\}$, with~$\chi(\eta)=1$ for~$|\eta|\leq1$, and we write
\begin{align*}
K_{0,0}(t,x)
    & =\mathfrak{F}^{-1} \big(\chi\mathfrak{F}(\tilde K_0)(t,\cdot)\big)=\mathfrak{F}^{-1} \big(\chi\tilde\varphi_0(t,\cdot)\,\sinc(\tilde\omega(t,\cdot))\big),\\
K_{0,1}(t,x)
    & =\mathfrak{F}^{-1} \big((1-\chi)\mathfrak{F}(\tilde K_0)(t,\cdot)\big)=\mathfrak{F}^{-1} \big((1-\chi)\tilde\varphi_0(t,\cdot)\,\sinc(\tilde\omega(t,\cdot))\big).
\end{align*}
To study~$K_{0,0}$ we may proceed as we did in Proposition~\ref{ppsmalltime}.
\begin{lemma}\label{lem:K00}
For any~$t\geq1$, it holds~$K_{0,0}(t,\cdot)\in L^1\cap L^\infty$ and
\[ \|K_{0,0}(t,\cdot)\|_{L^1}+\|K_{0,0}(t,\cdot)\|_{L^\infty} \leq C, \]
uniformly with respect to~$t\geq1$.
\end{lemma}
By Young inequality, Lemma~\ref{lem:K00} implies that
\[ \|K_{0,0}(t,\cdot)\ast f\|_{L^q} \leq C\,\|f\|_{L^p}, \]
for any~$t\geq1$ and~$1\leq p\leq q\leq\infty$.
\begin{proof}
We may follow the proof of Proposition~\ref{ppsmalltime}, but we may avoid the time dependence in the estimate. It is clear that~$|\hat K_{0,0}(t,\cdot)|\leq C$. On the other hand, as we did in~\eqref{eq:Taylorsin}, by Taylor's formula,
\begin{equation}\label{eq:Taylor}
\sinc \tilde\omega = \tilde\omega^{-1}\sin(\tilde\omega) = 1 - \frac1{6}\,\tilde\omega^2\,\int_0^1 (1-\rho)^3\sin(\rho\tilde\omega)\,d\rho,
\end{equation}
we easily derive
\[ \forall\gamma\neq0: \qquad |\partial_\eta^\gamma \sinc(\tilde\omega)| \lesssim |\eta|^{2\sigma-|\gamma|}, \]
for any~$|\eta|\leq 2$. Therefore,
\[ \forall \gamma: \qquad |\partial_\eta^\gamma \hat K_{0,0}(t,\eta)| \lesssim \big( 1+ |\eta|^{2\sigma-|\gamma|}\big). \]
If~$2\sigma>1$, we may now apply the first part of Lemma~\ref{lem:parts} with~$\kappa=n+1$, obtaining~$|K_{0,0}(t,x)|\leq C(1+|x|)^{-n-1}$, and this concludes the proof. If~$2\sigma\in(0,1]$, then we may apply the second part of Lemma~\ref{lem:parts} with~$\kappa=n$, $a=n-2\sigma$ and~$a_1=a+1$, obtaining
\[ |K_{0,0}(t,x)|\leq \begin{cases}
C(1+|x|)^{-n-2\sigma} & \text{if~$2\sigma\in(0,1)$,} \\
C(1+|x|)^{-n-1}\,\log(e+|x|) & \text{if~$2\sigma=1$,}
\end{cases} \]
and this concludes the proof.
\end{proof}
\begin{remark}
We may now study in details the part of the fundamental solution~$K_{0,1}(t,x)$, which is the most interesting one. We remark that $K_{0,1}$ has been localized in frequencies in two steps, i.e., first choosing low frequencies with respect to~$\xi$, and then choosing high frequencies with respect to~$\eta$. This two-steps localization in frequencies corresponds to localize the fundamental solution in the extended phase space, namely, $t^{-\frac1\sigma}\leq \xii\leq \eps_0$, i.e., $1\leq |\eta|\leq t^{\frac1\sigma}\eps_0$.

In this zone of the extended phase space we may employ the strategy used in~\cite{EL} (see also~\cite{Sjo}) to study the damping-free problem, replacing the homogeneity of the equation by an analogous, weaker property for the localized solution of our problem. Roughly speaking, in this zone of the extended phase space, our fundamental solution may be expressed by a scale-invariant term plus reminder terms. This is possible, since we already dropped the diffusive part of the equation, which possesses a different scaling. In some sense, by splitting the kernels of the fundamental solution and by a change of variable, we may ``mimic the homogeneity argument'' employed for the damping-free equation in~\cite{EL}, at least in the most important zone of the extended phase space, that is, $t^{-\frac1\sigma}\leq \xii\leq \eps_0$.
\end{remark}
\begin{proposition}\label{StationaryPhase}
We denote
\[ m(t,\eta)=\hat K_{0,1}(t,\eta)=(1-\chi(\eta))\varphi_0(t^{-\frac{1}{\sigma}}\eta)\sinc(\tilde\omega(t,\eta)). \]
Assume that~$1\leq p\leq q\leq \infty$ verify
\begin{align}
\label{eq:rangeder1th}
\frac{1-\sigma}{p}- \frac1{q}  < \sigma\left( \frac1{n}-\frac1{2}\right), &\qquad \text{if}\quad \frac1{p}+ \frac1{q}\leq 1,  \\
\label{eq:rangeder2th}
\frac1{p}+\frac{\sigma-1}{q}  <\sigma\left( \frac1{n}+\frac1{2}\right), & \qquad \text{if}\quad \frac1{p}+ \frac1{q}\geq 1.
\end{align}
Then $m(t, \cdot)\in M_p^q$ for any~$t\geq1$, and~$\|m(t,\cdot)\|_{M_p^q}\leq C$, that is,
\begin{equation}\label{eq:estK01}
\|K_{0,1}(t,\cdot)\ast f\|_{L^q} \leq C\,\|f\|_{L^p},
\end{equation}
uniformly with respect to~$t\geq1$. Moreover, if equality holds in~\eqref{eq:rangeder1th} or, respectively, \eqref{eq:rangeder2th}, estimate~\eqref{eq:estK01} remains valid with a possible log-loss, that is,
\begin{equation}\label{eq:logK01}
\|K_{0,1}(t,\cdot)\ast f\|_{L^q} \leq C\,\log(e+t)\,\|f\|_{L^p},
\end{equation}
for~$t\geq1$, where~$C>0$ does not depend on~$t$. If the equality holds in~\eqref{eq:rangeder1th} or, respectively, \eqref{eq:rangeder2th}, and~$1<p\leq 2\leq q<\infty$, the log-loss may be avoided, that is, we get again~\eqref{eq:estK01}.
\end{proposition}
\begin{proof}
We recall that~$(1-\chi(\eta))\varphi_0(t^{-\frac{1}{\sigma}}\eta)$ is supported in~$\{1\leq |\eta|\leq \eps_0t^{\frac1\sigma}\}$. Moreover,
\[ (1-\chi(\eta))\varphi_0(t^{-\frac{1}{\sigma}}\eta)=1, \quad \text{if}\quad 2\leq |\eta|\leq 2^{-1}\eps_0t^{\frac1\sigma}.\]
By using duality arguments, it is sufficient to prove Proposition \ref{StationaryPhase} for  $\frac{1}{p}+\frac{1}{q}\geq 1$.

Now let us consider a dyadic partition of unity~$\{\psi_k\}_{k\in\Z}$ as in Notation~\ref{not:Besov}. Due to~$\supp\psi_k \subset \{2^{k-1}\leq|\eta|\leq2^{k+1}\}$, if we define
\[ k_0=k_0(t)=\max\{k\in\Z: \ 2^k\leq \eps_0 t^{\frac1\sigma} \}, \]
we now see that (for sufficiently large~$t$):
\[ m(t,\eta)\psi_k(\eta) = \begin{cases}
0 & \text{if~$k\leq-1$,}\\
(1-\chi)\,\psi_k\,\sinc\tilde\omega & \text{if~$k=0,1$}\\
\psi_k\,\sinc\tilde\omega & \text{if~$2\leq k\leq k_0-2$,}\\
\tilde\varphi\,\psi_k\,\sinc\tilde\omega & \text{if~$k=k_0-1,k_0,k_0+1$,}\\
0 & \text{if~$k\geq k_0+2$.}
\end{cases} \]
In particular,
\[ m(t,\eta) = \sum_{k=0}^{k_0(t)+1} \psi_k(\eta)\,m(t,\eta), \]
so that
\[ \|m(t,\cdot)\|_{M_p^q} \leq \sum_{k=0}^\infty \|\psi_km(t,\cdot)\|_{M_p^q}, \]
uniformly with respect to~$t$. We immediately obtain
\begin{equation}\label{eq:M2}
\|\psi_km(t,\cdot)\|_{M_2} = \|\psi_km(t,\cdot)\|_{L^\infty} \leq C\max_{2^{k-1}\leq |\eta|\leq 2^{k+1}} |\eta|^{-\sigma} =C 2^{-(k-1)\sigma}=C_1\,2^{-k\sigma}.
\end{equation}
On the other hand, by
\[ \|\partial_\eta^\gamma (\psi_k m(t,\cdot))\|_{L^2}\leq C \Big(\int_{2^{k-1}\leq|\eta|\leq 2^{k+1}} |\eta|^{-2\sigma+2(\sigma-1)|\gamma|}\,d\eta \Big)^{\frac12} \leq C_1\,2^{k(\frac{n}{2}-\sigma+|\gamma|(\sigma-1))} \]
we derive, choosing some~$N>n/2$ (see Theorem~\ref{Th: Bernstein}), the estimate
\begin{equation}\label{eq:M1}
\|\psi_km(t,\cdot)\|_{M_1} \leq \|\psi_km(t,\cdot)\|_{L_2}^{1-\frac{n}{2N}}\,\sum_{|\gamma|=N}\|\partial_\eta^\gamma (\psi_k m(t,\cdot))\|_{L^2}^{\frac{n}{2N}} \leq C_2\,2^{k\sigma(\frac{n}{2}-1)}.
\end{equation}
Let now~$k=2,\ldots,k_0-2$. Since
\[m(t, \eta)\psi_k(\eta)=\left(e^{i\tilde\omega(t,\eta)} - e^{-i\tilde\omega(t,\eta)}\right)\frac{\psi_k(\eta)}{2i\tilde\omega(t,\eta)}\]
replacing $\eta$ by $2^{k}  \eta$ and by using that $\tilde\omega(t,2^{k}\eta)=2^{k\sigma}\tilde\omega(2^{-k\sigma}t,\eta)\approx 2^{k\sigma} |\eta|^\sigma$ and Littman's lemma (Lemma \ref{ThmLittmanlemmapecher}) we conclude
\begin{align}
\nonumber
\Big\|\mathfrak{F}^{-1}_{\eta\rightarrow x}\Big(e^{\pm i\tilde\omega(t,\eta)}\frac{\psi_k(\eta)}{\tilde\omega(t,\eta)}\Big)\Big\|_{L^{\infty}(\R^n)}&=2^{k(n-\sigma)}\Big\|\mathfrak{F}^{-1}_{\eta\rightarrow x}\Big(e^{\pm i2^{k\sigma}\tilde\omega(2^{-k\sigma}t,\eta)}\frac{\psi(\eta)}{\tilde\omega(2^{-k\sigma}t,\eta)}\Big)\Big\|_{L^{\infty}(\R^n)}  \\
\label{eq:littman}
&\leq C 2^{k(n-\sigma)}(1+2^{k\sigma})^{-\frac{n}{2}} \leq C2^{k\left(n-\sigma-\frac{n}2\sigma\right)},
\end{align}
for all $k=2,\ldots,k_0-2$. We used that for $\sigma\neq 1$ the rank of the Hessian $H_{\tilde\omega(2^{-k\sigma}t,\eta)}$ is $n$, and that, for sufficiently small~$\eps_0$, it holds
\begin{equation}\label{eq:Hessian}
|\det H_{\tilde\omega(2^{-k\sigma}t,\eta)}|\geq c_{n,\sigma}>0, \quad \text{uniformly with respect to~$t$.}
\end{equation}
Indeed,
\[ H_{\tilde\omega(2^{-k\sigma}t,\eta)} = H_{|\eta|^\sigma} + \textit{O}(\eps_0), \]
due to~\eqref{eq:perturbation}. We emphasize that it is not possible to extend this approach to the case~$\sigma=1$, due to~$\det H_{|\eta|}=0$, so that we cannot get~\eqref{eq:Hessian}.

By Young's convolution inequality, we get
\begin{equation}\label{eq:M1inf}
\|m\psi_k\|_{M_1^{\infty}}\leq C2^{k(n-\sigma-\frac{n}2\sigma)}.
\end{equation}
The same holds true for~$k=0,1,k_0-1,k_0,k_0+1$, possibly modifying the constant~$C$.

As a consequence of Riesz-Thorin interpolation theorem, by~\eqref{eq:M2} and~\eqref{eq:M1inf}, we get
\begin{equation}\label{eq:Mpconj}
\|m\psi_k\|_{M_{p_0}^{q_0}}\leq C2^{k\Big(-\sigma+(\frac{1}{p_0}-\frac{1}{2})(n(2-\sigma))\Big)}
\end{equation}
for~$p_0$, $q_0$ on the conjugate line, that is, $\frac{1}{p_0}+\frac{1}{q_0}=1$.

Using \eqref{eq:M1}, \eqref{eq:Mpconj} and  Riesz-Thorin interpolation theorem we conclude that
\begin{align*}
 \|m\psi_k\|_{M_{p}^{q}}\leq  C2^{k\Big(-\sigma+(\frac{1}{p_0}-\frac{1}{2})(n(2-\sigma))\Big)(1-\delta)}2^{k\sigma(\frac{n}{2}-1)\delta}
 =C2^{kn\Big(\frac{1}{p}+\frac{\sigma-1}{q}-\big(\frac{\sigma}{2}+\frac{\sigma}{n}\big)\Big)},
\end{align*}
where $0< \delta < 1$, with $\frac{1}{p}=\frac{1-\delta}{p_0}+\delta$ and $\frac{1}{q}=\frac{1-\delta}{q_{0}}+\delta$.

Therefore, we conclude the estimate
\[
\|m\|_{M_{p}^{q}} \leq \sum_{k=0}^{k_0(t)+1}\|m\psi_k\|_{M_{p}^{q}} \leq C\,\sum_{k=0}^\infty 2^{kn\Big(\frac{1}{p}+\frac{\sigma-1}{q}-\big(\frac{\sigma}{2}+\frac{\sigma}{n}\big)\Big)},
\]
uniformly with respect to~$t\geq1$ (since we removed the bound from above~$k_0(t)+1$ on the indexes).

The latter series converges if, and only if, \eqref{eq:rangeder2th} holds. If the equality holds in~\eqref{eq:rangeder2th}, we modify the proof using the definition of~$k_0(t)$ to obtain a dependence on~$t$:
\[
\|m\|_{M_{p}^{q}} \leq \sum_{k=0}^{k_0(t)+1}\|m\psi_k\|_{M_{p}^{q}} \leq C\,(k_0(t)+2) \leq C\big(2+\log_2 t^{\frac1\sigma}\big) \lesssim \log(e+t),
\]
and this concludes the proof.

However, in the special case~$1<p\leq 2\leq q<\infty$, the latter estimate may be refined by using the embeddings for Besov spaces (see, for instance, \cite{ST}): $L^p\hookrightarrow B^0_{p,2}$ for~$p\in(1,2]$ and $B^0_{q,2}\hookrightarrow L^q$ for~$q\in[2,\infty)$. Indeed, since the sum in~\eqref{eq:partition} is finite for any given~$\xi$, in particular, $\#\{k: \ \psi_k(\xi)\neq0\}\leq 3$, we obtain the chain of inequality (see also~\cite{Brenner1})
\[ \|\mathfrak{F}^{-1}(m\hat f)\|_{B^0_{q,2}}\leq C_1\sup_k \|\mathfrak{F}^{-1}(m\psi_k\hat f)\|_{L^q} \leq C_2\,\|f\|_{L^p}\leq C_3 \,\|f\|_{B^0_{p,2}}, \]
and this concludes the proof.
\end{proof}
\begin{remark}\label{rem:alternative}
We notice that $1/p+1/q\geq1$, together with~$p\leq q$, is equivalent to ask that~$p\leq2$ and~$p\leq q\leq p'$, whereas $1/p+1/q\leq1$, together with~$p\leq q$, is equivalent to ask that~$q\geq2$ and~$q'\leq p\leq q$. Here by~$p'$ we denote the H\"older conjugate of~$p$, i.e., $1/p+1/p'=1$.

Then we may rewrite~\eqref{eq:rangeder1th} as
\begin{align}
\label{eq:range1equiv}
\frac{n}\sigma\left(\frac1p-\frac1q\right) + n\left(\frac12-\frac1p\right)<1, &\qquad  \text{if~$q\in[2,\infty]$ and~$p\in[q',q]$,}
\intertext{and we may rewrite~\eqref{eq:rangeder2th} as}
\label{eq:range2equiv}
\frac{n}\sigma\left(\frac1p-\frac1q\right) + n\left(\frac1q-\frac12\right)<1, &\qquad  \text{if~$p\in[1,2]$ and~$q\in[p,p']$.}
\end{align}
Noticing that
\[ q'\leq p \iff \frac12-\frac1p \geq \frac1q-\frac12, \]
and
\[ q\leq p' \iff \frac1q-\frac12 \geq \frac12-\frac1p, \]
we may rewrite both~\eqref{eq:range1equiv} and~\eqref{eq:range2equiv} as~\eqref{eq:rangeboth}.
\end{remark}
As a consequence of Proposition~\ref{ppsmalltime}, Lemma~\ref{lem:K00} and Proposition~\ref{StationaryPhase}, we have proved the following.
\begin{proposition}\label{main}
Let $\sigma>1$. Assume that~$1\leq p\leq q\leq \infty$ verify~\eqref{eq:rangeboth}. Then we have the following $L^p-L^q$ estimate
\begin{eqnarray}\label{linearestimates}
\| K_0(t,\cdot)\ast u_1\|_{L^q} \lesssim\, (1+t)^{1-\frac{n}{\sigma}\left( \frac1{p}- \frac1{q}\right)}\,\|u_1||_{L^p}, \qquad  \forall t\geq 0.
\end{eqnarray}
Moreover, if equality holds in~\eqref{eq:rangeboth}, estimate~\eqref{linearestimates} remains valid with a possible log-loss, that is,
\begin{equation}\label{eq:logK0}
\|K_0(t,\cdot)\ast u_1\|_{L^q} \leq C\,(1+t)^{1-\frac{n}{\sigma}\left( \frac1{p}- \frac1{q}\right)}\,\log(e+t)\,\|u_1\|_{L^p},\qquad  \forall t\geq 0,
\end{equation}
or, equivalently,
\begin{equation}
\label{eq:logK0equiv}
\|K_0(t,\cdot)\ast u_1\|_{L^q}\leq C\,(1+t)^{n\max\left\{\left(\frac12-\frac1p\right),\left(\frac1q-\frac12\right)\right\}}\,\log(e+t)\,\|u_1\|_{L^p},\qquad  \forall t\geq 0.
\end{equation}
If the equality holds in~\eqref{eq:rangeboth}, and~$1<p\leq 2\leq q<\infty$, the log-loss may be avoided, that is, we get again~\eqref{linearestimates}.
\end{proposition}

\section{$L^p-L^q$ low frequencies estimates ($\xii\leq\eps_0$) for derivatives of the solution}\label{sec:lowder}

The extension of Proposition~\ref{main} to include classical derivatives~$\partial_x^\beta$ of the solution, fractional derivatives~$(-\Delta)^{\frac{b}2}$ and time derivatives~$\partial_t^\ell$, is pretty much straightforward, so we postponed this analysis here, for the ease of reading.

The extension of Proposition~\ref{ppsmalltime} requires a few minor modifications in the proof.
\begin{proposition}\label{ppsmalltimefract}
Let $T\geq1$. Then, for any $t\in[0,T]$ and $1\leq p\leq q\leq \infty$, it holds
\[
\norm{\partial_t^\ell\partial_x^\beta (-\Delta)^{\frac{b}2} K_0(t,\cdot)\ast u_1}_{L^q}\leq C(T) \norm{u_1}_{L^p},
\]
for any~$\ell\in\N$, $\beta\in\N^n$, $b\geq0$,for some~$C(T)$, independent of~$u_1$.
\end{proposition}
\begin{proof}
As in the proof of Proposition~\ref{ppsmalltime}, it is sufficient to apply Lemma~\ref{lem:parts} to
\[ f(t,\xi)=(i\xi)^\beta|\xi|^{b}\varphi_0 \partial_t^\ell\hat K(t,\cdot). \]
We notice that
\[ \partial_t^\ell\hat K_0(t,\cdot) = \varphi_0 \sum_{j=0}^\ell\, \binom{\ell}{j}\, \omega^{j-1}\sin^{(j)}(t\omega) \partial_t^{\ell-j} (e^{-t\xii^{2\theta}/2}), \]
where~$\sin^{(j)}(t\omega)=(-1)^{j/2}\sin(t\omega)$ for~$j$ even and $\sin^{(j)}(t\omega)=(-1)^{(j-1)/2}\cos(t\omega)$ for~$j$ odd.

It is clear that~$|f(t,\xi)|\leq C$. To estimate its derivatives, we now consider two cases. First, let~$b=0$.

By virtue of~\eqref{eq:Taylorsin}, and its analogous for the cosine function,
\begin{equation}
\label{eq:Taylorcos}
\cos(t\omega) = 1 - \frac1{2}\,t^2\omega^2\,\int_0^1 (1-\rho)^2\cos(\rho\omega t)\,d\rho,
\end{equation}
we easily derive
\[ \forall\gamma\neq0: \qquad |\partial_\xi^\gamma (\omega^{j-1}\sin^{(j)}(t\omega))| \lesssim (1+t^3)\,\xii^{2\sigma-|\gamma|}, \]
for any~$\xii\leq \eps_0$. On the other hand,
\[ \forall\gamma\neq0: \qquad |\partial_\xi^\gamma \partial_t^{\ell-j} (e^{-t\xii^{2\theta}/2})| \lesssim t\,\xii^{2\theta-|\gamma|}. \]
Therefore, as in the proof of Proposition~\ref{ppsmalltime}, we get (here we use that~$(i\xi)^\beta$ is smooth)
\[ \forall \gamma: \qquad \big|\partial_\xi^\gamma f(t,\xi)\big| \lesssim C(T)\,\big( 1+ \xii^{2\theta-|\gamma|}\big), \]
and so we conclude the proof applying Lemma~\ref{lem:parts}.

Now let~$b>0$. The proof is trivial if~$q\in(1,\infty)$, due to
\[ \|(-\Delta)^{\frac{b}2}g\|_{L^q} \leq \|g\|_{W^{m,q}}, \]
for any~$m\geq b$. However, if~$q=1$ or~$q=\infty$, the estimate above may fail, in general. Therefore, we modify the proof, to take into account of possibly fractional values of~$b$.

Now the use of Taylor formula to get~\eqref{eq:Taylorsin} and~\eqref{eq:Taylorcos} is no longer helpful, since~$\xii^b$ is not a smooth function. However, it is clear that
\[ \forall \gamma: \qquad |\partial_\xi^\gamma f(t,\xi)| \lesssim C(T)\,\xii^{b-|\gamma|}. \]
If~$b>1$, we apply the first Lemma~\ref{lem:parts} with~$\kappa=n+1$ and~$a=n+1-b$, obtaining~$|g(t,\xi)|\leq C(T)\,(1+|x|)^{-n-1}$. If~$b\in(0,1]$, we apply the second part of Lemma~\ref{lem:parts} with~$\kappa=n$ and~$a=n-b$ and~$a_1=a+1$, obtaining
\[ |g(t,\xi)|\leq \begin{cases}
C\,(1+|x|)^{-n-b} & \text{if~$b\in(0,1)$,} \\
C\,(1+|x|)^{-n-1}\,\log(e+|x|) & \text{if~$b=1$,}
\end{cases} \]
and concluding the proof.
\end{proof}
We now replace~\eqref{split} by
\begin{equation}\label{eq:splitder}
\|\partial_x^\beta (-\Delta)^{\frac{b}2}\partial_t^\ell K_0(t,\cdot)\|_{L^q_p}\leq Ct\,\|\partial_x^\beta (-\Delta)^{\frac{b}2}\partial_t^\ell \mathscr{F}^{-1}(\sinc(\omega t)\varphi_0)\|_{L^q_p}.
\end{equation}
Again, we are now legitimated to perform the change of variable $\eta=t^{\frac{1}{\sigma}}\xi$, for~$t\geq1$, which gives
\begin{equation}
\| \partial_x^\beta (-\Delta)^{\frac{b}2}\partial_t^\ell \mathscr{F}^{-1}(\sinc(\omega t)\varphi_0)\|_{L^q_p}= t^{-\frac{n}{\sigma}(\frac{1}{p}-\frac{1}{q})-\frac{|\beta|+b}\sigma-\ell}
\|  \partial_x^\beta (-\Delta)^{\frac{b}2}\partial_t^\ell \tilde K_0(t,\cdot)\|_{L^q_p}.
\end{equation}
As we did for Proposition~\ref{ppsmalltimefract}, we may extend Lemma~\ref{lem:K00} to cover the case of derivatives and fractional derivatives. To extend Proposition~\ref{StationaryPhase}, we shall only take into account of the influence of the derivatives, which leads to obtain an additional power $2^{k(|\beta|+b+\sigma\ell)}$ in all estimates for~$\psi_k(\eta)m(t,\eta)$. In turn, we obtain
\[
\|m\|_{M_{p}^{q}} \leq C\,\sum_{k=0}^{k_0(t)+1} 2^{kn\Big(\frac{1}{p}+\frac{\sigma-1}{q}-\big(\frac{\sigma}{2}+\frac{\sigma}{n}\big)\Big)+k(|\beta|+b+\sigma\ell)}.
\]
As in the proof of Proposition~\ref{StationaryPhase}, the sum is bounded by a constant~$C$, uniformly with respect to~$t\geq1$, if we assume
\begin{align}
\label{eq:rangeder2}
\frac1{p}+\frac{\sigma-1}{q} + \frac{|\beta|+b}{n} <\sigma\left( \frac{1-\ell}{n}+\frac1{2}\right), & \qquad \text{if} \quad \frac1{p}+ \frac1{q}\geq 1,
\intertext{whereas a log-loss appears if we take the equality in~\eqref{eq:rangeder2}. Its dual condition is}
\label{eq:rangeder1}
\frac{1-\sigma}{p}- \frac1{q} + \frac{|\beta|+b}{n} < \sigma\left( \frac{1-\ell}{n}-\frac1{2}\right), &\qquad  \text{if}\quad \frac1{p}+ \frac1{q}\leq 1.
\end{align}
As we did in Remark~\ref{rem:alternative}, we may write~\eqref{eq:rangeder2} and~\eqref{eq:rangeder1} as a unique condition (see~\eqref{eq:rangederboth}).

Hence, we obtain the following generalization of Proposition~\ref{main}.
\begin{theorem}\label{thm:lowder}
Let $\sigma>1$, $\ell\in\N$, $\beta\in\N^n$ and $b\geq0$. Assume that~$1\leq p\leq q\leq \infty$ and that
\begin{equation}\label{eq:rangederboth}
\frac{n}\sigma\left(\frac1p-\frac1q\right)+n\max\left\{\left(\frac12-\frac1p\right),\left(\frac1q-\frac12\right)\right\} +\frac{|\beta|+b}{\sigma} <1-\ell.
\end{equation}
Then we have the following $L^p-L^q$ estimate
\begin{eqnarray}\label{linearenergyestimates}
\| \partial_x^\beta (-\Delta)^{\frac{b}2}\partial_t^\ell  K_0(t,\cdot)\ast u_1\|_{L^q} \lesssim\, (1+t)^{1-\frac{n}{\sigma}\left( \frac1{p}- \frac1{q}\right)-\frac{|\beta|+b}{\sigma}-\ell}\,\|u_1\|_{L^p}, \qquad  \forall t\geq 0.
\end{eqnarray}
Moreover, if equality holds in~\eqref{eq:rangederboth}, estimate~\eqref{linearenergyestimates} remains valid with a possible log-loss, that is,
\begin{equation}\label{eq:logK0der}
\| \partial_x^\beta (-\Delta)^{\frac{b}2}\partial_t^\ell  K_0(t,\cdot)\ast u_1\|_{L^q} \lesssim\, (1+t)^{n\max\left\{\left(\frac12-\frac1p\right),\left(\frac1q-\frac12\right)\right\}}\,\log(e+t)\,\|u_1\|_{L^p}, \qquad  \forall t\geq 0.
\end{equation}
If the equality holds in~\eqref{eq:rangederboth}, and~$1<p\leq 2\leq q<\infty$, the log-loss may be avoided, that is, we get again~\eqref{linearenergyestimates}.
\end{theorem}
\begin{remark}
We have a special interest into obtain $L^p-L^q$ estimates at the ``energy level'', that is, for~$(-\Delta)^{\frac\sigma2}u$ and for~$u_t$. Setting~$\beta=0$ and~$(b,\ell)=(\sigma,0),(0,1)$, we obtain the $L^p-L^q$ estimate
\begin{eqnarray}\label{linearenergy}
\|(-\Delta)^{\frac{\sigma}2} K_0(t,\cdot)\ast u_1\|_{L^q}+ \|\partial_t K_0(t,\cdot)\ast u_1\|_{L^q}\lesssim\, (1+t)^{-\frac{n}{\sigma}\left( \frac1{p}- \frac1{q}\right)}\,\|u_1\|_{L^p(\mathbb{R}^n)}, \qquad  \forall t\geq 0,
\end{eqnarray}
provided that
\begin{equation}
\label{eq:rangeENboth}
\frac1p-\frac1q <\sigma\,\min\left\{\left(\frac1p-\frac12\right),\left(\frac12-\frac1q\right)\right\}.
\end{equation}
A log loss appears if we take the equality in~\eqref{eq:rangeENboth}, unless~$1<p\leq 2\leq q<\infty$.

We immediately see that condition~\eqref{eq:rangeENboth} is not satisfied if~$\sigma\leq2$. Condition~\eqref{eq:rangeENboth} is verified on the conjugate line~$1/p+1/q=1$, for any~$\sigma>2$, exception given for~$p=q=2$. Moreover, away from the conjugate line~$1/p+1/q=1$, condition~\eqref{eq:rangeENboth} may only be satisfied if~$p<2<q$.
\end{remark}

\section{What happens if we do not split the diffusive part and the oscillating part}\label{sec:whatif}

In order to show the efficiency of the estimates obtained in Proposition~\ref{main} and Theorem~\ref{thm:lowder}, we compare our decay estimates with the result obtained by estimating the fundamental solution~$K_0(t,x)$ in low-frequencies solution~$K_0(t,\cdot)\ast u_1$, without isolating the diffusive and oscillating part.

For the sake of brevity, we only consider the easier case~$1\leq p\leq 2\leq q\leq\infty$. Let~$\ell\in\N$ and~$\beta\in\N^n$, $b\geq0$ be such that
\begin{equation}\label{eq:notsingbasic}
\text{either~$\ell\geq1$, or}\ |\beta|+b+n\left(\frac1p-\frac1q\right)>\sigma,\ \text{if~$p<2<q$, and~$|\beta|+b\geq\sigma$ if~$p=q=2$.}
\end{equation}
Then,
\begin{equation}\label{eq:classical}
\| \partial_x^\beta (-\Delta)^{\frac{b}2}\partial_t^\ell K_0(t,\cdot)\ast u_1\|_{L^q} \leq C\,(1+t)^{\frac\sigma{2\theta}-\frac{n}{2\theta}\left(\frac1p-\frac1q\right)-\frac{|\beta|+b}{2\theta}-\frac\sigma{2\theta}\,\ell}\,\|u_1\|_{L^p}.
\end{equation}
Indeed, by Haussdorff-Young inequality and H\"older inequality, setting
\[ \frac1r = \frac1{q'}-\frac1{p'}=\frac1p-\frac1q, \]
one may estimate, for~$t\geq1$,
\begin{align*}
\| \partial_x^\beta (-\Delta)^{\frac{b}2}\partial_t^\ell K_0(t,\cdot)\ast u_1\|_{L^q}
    & \lesssim \| (i\xi)^\beta \xii^b \partial_t^\ell  \hat K_0(t,\cdot)\hat u_1\|_{L^{q'}} \lesssim \| (i\xi)^\beta \xii^b \partial_t^\ell  \hat K_0(t,\cdot)\|_{L^r}\|\hat u_1\|_{L^{p'}} \\
    & \lesssim \| \xii^{|\beta|+b+(\ell-1)\sigma}\,e^{-t\xii^{2\theta}}\|_{L^r}\,\|u_1\|_{L^p} = C\,t^{-\frac{|\beta|+b+(\ell-1)\sigma+n/r}{2\theta}} \|u_1\|_{L^p}\,,
\end{align*}
so that we obtain~\eqref{eq:classical}.

The decay rate obtained in~\eqref{eq:classical} is worse than the one provided by~\eqref{linearestimates} and~\eqref{linearenergyestimates}, due to~$\sigma<2\theta$ and~\eqref{eq:notsingbasic}.

On the other hand, if~\eqref{eq:notsingbasic} is violated, so that $\xii^{|\beta|+b-\sigma}$ is not in~$L^r$, we obtain the same estimate in~\eqref{linearestimates} and~\eqref{linearenergyestimates}. Indeed, using
\[ |(i\xi)^\beta \xii^b \hat K_0(t,\cdot)|\leq \begin{cases}
Ct^{\,a+1-\frac{n/r+|\beta|+b}\sigma} \xii^{-\frac{n}r+a\sigma}, & \text{if~$\xii\leq t^{-\frac1\sigma}$,} \\
C\xii^{|\beta|+b-\sigma} & \text{if~$\xii\in [t^{-\frac1\sigma},\eps_0]$,}
\end{cases} \]
for a sufficiently small~$a\in(0,1)$, we now get
\begin{align*} \|(i\xi)^\beta \xii^b \hat K_0(t,\cdot)\|_{L^r}
     & \leq C\,t^{\,a+1-\frac{n/r+|\beta|+b}\sigma} \Big(\int_{\xii\leq t^{-\frac1\sigma}} \xii^{-n+a\sigma r}\, d\xi \Big)^{\frac1r} + C\,\Big(\int_{t^{-\frac1\sigma}\leq\xii\leq\eps_0} \xii^{(|\beta|+b-\sigma)r}\,d\xi \Big)^{\frac1r} \\
     & \approx \begin{cases}
    t^{1-\frac{n/r+|\beta|+b}\sigma}, & \text{if~$|\beta|+b-\sigma<n/r$,}\\
    (\log (e+t))^{\frac1p-\frac1q}, & \text{if~$|\beta|+b-\sigma=n/r$,}
    \end{cases}
\end{align*}
for~$r<\infty$, whereas we estimate $|(i\xi)^\beta \xii^b \hat K_0(t,\cdot)|\leq Ct^{1-\frac{|\beta|+b}\sigma}$ for~$r=\infty$.

Therefore, we obtain
\begin{equation}\label{eq:classical2}
\| \partial_x^\beta (-\Delta)^{\frac{b}2}\partial_t^\ell K_0(t,\cdot)\ast u_1\|_{L^q} \leq C\,(1+t)^{1-\frac{n}\sigma\left(\frac1p-\frac1q\right)-\frac{|\beta|+b}\sigma}\,\|u_1\|_{L^p},
\end{equation}
that is, the same of~\eqref{linearestimates} and~\eqref{linearenergyestimates}, if~$|\beta|+b+n\left(\frac1p-\frac1q\right)<\sigma$.

In turns, this implies that estimates~\eqref{linearestimates} and~\eqref{linearenergyestimates} improves the estimates obtained without splitting the kernels when~$1\leq p\leq 2\leq q\leq \infty$ if, and only if, $(p,q)\neq (2,2)$ and
\begin{equation}\label{eq:optimaldecayrange}
0\leq \frac{n}\sigma\left(\frac1p-\frac1q\right) +\frac{|\beta|+b}{\sigma} + \ell-1 < n\min\left\{\left(\frac1p-\frac12\right),\left(\frac12-\frac1q\right)\right\}.
\end{equation}
However, we shall mention that the case in which
\[ \frac{n}\sigma\left(\frac1p-\frac1q\right) +\frac{|\beta|+b}{\sigma} + \ell-1<0 \]
is of minor interest since, in this case, the estimate provided by~\eqref{linearestimates} and~\eqref{linearenergyestimates} or, equivalently, by~\eqref{eq:classical2}, does not produce a decay rate, but only a control on the possible increasing behavior of the norm as~$t$ grows.

Our approach improves the estimates that may be obtained without splitting the kernels, even for~$1\leq p\leq q<2$ and~$2<q\leq p\leq\infty$, due to the fact that the oscillations lead to an extra loss of decay rate in this case (see, for instance, \cite{DAEP16}) if the kernels are not split, but we avoid the details for the sake of brevity.

Finally, we anticipate that our approach also improves the decay rate in~\eqref{eq:classical} if~\eqref{eq:notsingbasic} holds, provided that some condition, less restrictive than~\eqref{eq:optimaldecayrange}, holds, as shown in Section~\ref{sec:loss} (see Remark~\ref{rem:loss}).


\section{The loss of decay rate in $L^p-L^q$ estimates out of the optimal range for~$(p,q)$}\label{sec:loss}

If we are out of the $(p,q)$ range given by~\eqref{eq:rangederboth} in Theorem~\ref{thm:lowder}, then we may obtain a decay estimate, but a loss of decay appears, with respect to the case in which~\eqref{eq:rangederboth} holds. This situation is quite different with respect to the case of an evolution equation damping-free. Indeed, in such a case, $L^p-L^q$ estimates do not hold out of a $(p,q)$ range analogous to the one in Theorem~\ref{thm:lowder}. Thanks to the presence of the noneffective damping, we may still have estimates outside of these ranges, but we sacrifice some loss of decay, using the multiplier related to the diffusive part of the solution.

This loss becomes larger when~$\theta$ goes from~$\sigma/2$ to~$\sigma$, consistently with the fact that we have no loss in the limit case~$\theta=\sigma/2$. Indeed, the loss originates from the different scaling in the diffusive part of the multiplier, i.e. a $(1,2\theta)$ scaling for~$(t,x)$ in the diffusive part of the multiplier, and a~$(1,\sigma)$ scaling for~$(t,x)$ in the evolution part of the multiplier.

We want to prove the following.
\begin{theorem}\label{thm:loss}
Let $\sigma>1$, $\ell\in\N$, $\beta\in\N^n$ and $b\geq0$. Assume that~$1\leq p\leq q\leq \infty$ and that
\begin{equation}\label{eq:loss}
a = n\left(\frac1p-\frac1q\right)+|\beta|+b+\sigma \left( n\max\left\{\left(\frac12-\frac1p\right),\left(\frac1q-\frac12\right)\right\} +\ell-1\right)
\end{equation}
is nonnegative. Moreover, assume that
\begin{equation}\label{eq:wavereg}
n\max\left\{\left(\frac12-\frac1p\right),\left(\frac1q-\frac12\right)\right\}<1.
\end{equation}
Then we have the following $L^p-L^q$ estimate
\begin{eqnarray}\label{eq:LpLqloss}
\| \partial_x^\beta (-\Delta)^{\frac{b}2}\partial_t^\ell  K_0(t,\cdot)\ast u_1\|_{L^q} \lesssim\, (1+t)^{n\max\left\{\left(\frac12-\frac1p\right),\left(\frac1q-\frac12\right)\right\}-\frac{a}{2\theta}}\,\log(e+t)\,\|u_1\|_{L^p}, \qquad  \forall t\geq 0.
\end{eqnarray}
If~$1<p\leq 2\leq q<\infty$, the log-loss disappears, that is, \eqref{eq:LpLqloss} becomes
\begin{eqnarray}\label{eq:LpLqlossnolog}
\| \partial_x^\beta (-\Delta)^{\frac{b}2}\partial_t^\ell  K_0(t,\cdot)\ast u_1\|_{L^q} \lesssim\, (1+t)^{n\max\left\{\left(\frac12-\frac1p\right),\left(\frac1q-\frac12\right)\right\}-\frac{a}{2\theta}}\,\|u_1\|_{L^p}, \qquad  \forall t\geq 0.
\end{eqnarray}
The log-loss also disappears, that is, we obtain
\begin{eqnarray}\label{eq:LpLqlossnologspecial}
\| \partial_x^\beta (-\Delta)^{\frac{b}2}\partial_t^\ell  K_0(t,\cdot)\ast u_1\|_{L^q} \lesssim\, (1+t)^{-\frac{a}{2\theta}}\,\|u_1\|_{L^p}, \qquad  \forall t\geq 0.
\end{eqnarray}
if~$(p,q)=(1,2)$ or~$(p,q)=(2,\infty)$ and
\[ a=\frac{n}2+|\beta|+b+\sigma(\ell-1)>0. \]
\end{theorem}
The loss out of the optimal range~\eqref{eq:rangederboth}, appearing in Theorem~\ref{thm:loss}, is due to the fact that a term~$-a/(2\theta)$ appears in~\eqref{eq:LpLqloss}, in place of~$-a/\sigma$, and~$\sigma<2\theta$. In other words, the loss tends to vanish as~$\theta\to\sigma/2$, i.e., the model becomes closer to the effective damping case.
\begin{remark}
We notice that condition~\eqref{eq:wavereg} is trivially verified if~$1\leq p\leq2\leq q\leq \infty$. Otherwise, it reads as $n(1/q-1/2)<1$ if~$p\leq q\leq2$ or~$n(1/2-1/p)<1$ if~$2\leq p\leq q$. Condition~\eqref{eq:wavereg} is a trivial consequence of~\eqref{eq:rangederboth} in Theorem~\ref{thm:lowder}. Once again, this assumption is related to the restriction on~$L^q-L^q$ estimates or, respectively, $L^p-L^p$ estimates, for a evolution equation damping-free, so it looks natural that it cannot be dropped in a result based in exploiting the influence of the oscillatory part of the fundamental solution to obtain optimal $L^p-L^q$ estimates.
\end{remark}
\begin{proof}[Theorem~\ref{thm:loss}]
In order to prove \eqref{eq:LpLqloss} and~\eqref{eq:LpLqlossnolog}, for any~$t\geq1$, we now replace~\eqref{eq:splitder} by
\begin{equation}\label{eq:splitloss}
\|\partial_x^\beta (-\Delta)^{\frac{b}2}\partial_t^\ell K_0(t,\cdot)\|_{L^q_p}\leq C t\,\|\mathscr{F}^{-1}(\xii^a\,e^{-t|\xi|^{2\theta}/2})\|_{L^1}\,\|I_a\partial_x^\beta (-\Delta)^{\frac{b}2}\partial_t^\ell \mathscr{F}^{-1}(\sinc(\omega t)\varphi_0)\|_{L^q_p}.
\end{equation}
It is clear that
\[ \|\mathscr{F}^{-1}(\xii^a\,e^{-t|\xi|^{2\theta}/2})\|_{L^1}= t^{-\frac{a}{2\theta}}\,\|\mathscr{F}^{-1}(\xii^a\,e^{-|\xi|^{2\theta}/2})\|_{L^1},\]
for any~$t>0$. On the other hand,
\begin{align*}
\| I_a\partial_x^\beta (-\Delta)^{\frac{b}2}\partial_t^\ell \mathscr{F}^{-1}(\sinc(\omega t)\varphi_0)\|_{L^q_p}
    & = t^{-\frac{n}{\sigma}(\frac{1}{p}-\frac{1}{q})+\frac{a-|\beta|-b}\sigma-\ell}\|I_a \partial_x^\beta (-\Delta)^{\frac{b}2}\partial_t^\ell \tilde K_0(t,\cdot)\|_{L^q_p}\\
    & = t^{n\max\left\{\left(\frac12-\frac1p\right),\left(\frac1q-\frac12\right)\right\}-1}\,\|I_a \partial_x^\beta (-\Delta)^{\frac{b}2}\partial_t^\ell \tilde K_0(t,\cdot)\|_{L^q_p}.
\end{align*}
Therefore, the proof of Theorem~\ref{thm:loss} reduces to show that
\begin{align}
\label{eq:loss1}
\|I_a \partial_x^\beta (-\Delta)^{\frac{b}2}\partial_t^\ell K_{0,0}(t,\cdot)\|_{L^q_p}
    & \leq C, \\
\label{eq:loss2}
\|I_a \partial_x^\beta (-\Delta)^{\frac{b}2}\partial_t^\ell K_{0,1}(t,\cdot)\|_{L^q_p}
    & \leq \begin{cases}
    C & \text{if~$1<p\leq 2\leq q<\infty$,}\\
    C\,\log(e+t), & \text{otherwise,}
    \end{cases}
\end{align}
with~$C$ independent of~$t\geq1$. The difference with respect to the analysis in Section~\ref{sec:lowder} is related to the presence of the Riesz potential, so we shall guarantee that this influence may be managed without difficulties. For~\eqref{eq:loss2}, this is trivial, following the proof of Proposition~\ref{StationaryPhase}, as we did in Section~\ref{sec:lowder}. So let us consider~\eqref{eq:loss1}.

Estimate~\eqref{eq:loss1} is a mere consequence of Proposition~\ref{ppsmalltimefract} if~$a\leq b$, so we may assume that~$b<a$, and write~\eqref{eq:loss1} in the form
\[ \|I_{a-b} \partial_x^\beta\partial_t^\ell K_{0,0}(t,\cdot)\|_{L^q_p}\leq C, \]
which follows as a consequence of Young inequality and
\begin{equation}\label{eq:Lr}
\|I_{a-b} \partial_x^\beta\partial_t^\ell K_{0,0}(t,\cdot)\|_{L^r}\leq C,
\end{equation}
where
\[ 1-\frac1r = \frac1p-\frac1q. \]
We easily compute
\begin{equation}\label{eq:gammaloss}
\forall\gamma: \quad \big|\partial_\eta^\gamma \big( |\eta|^{-(a-b)}(i\eta)^\beta \partial_t^\ell \sinc \tilde\omega(t,\eta) \big)\big| \lesssim |\eta|^{-a+b+|\beta|+\ell\sigma-|\gamma|}=|\eta|^{\delta\sigma-n\left(1-\frac1r\right)-|\gamma|},
\end{equation}
for any~$|\eta|\leq2$, where we define
\[ \delta = 1 - n \max\left\{\left(\frac12-\frac1p\right),\left(\frac1q-\frac12\right)\right\}. \]
We remark that~$\delta>0$ if, and only if, \eqref{eq:wavereg} holds. Thanks to~$\delta>0$, we may apply Lemma~\ref{lem:parts} and derive~\eqref{eq:Lr}.

It only remains to prove~\eqref{eq:LpLqlossnologspecial}, but this latter corresponds to~\eqref{eq:classical}, which we already proved, and this concludes the proof.
\end{proof}
\begin{remark}\label{rem:loss}
To show that Theorem~\ref{thm:loss} still provide benefits coming from the strategy of splitting the kernel, we may compare estimate~\eqref{eq:LpLqloss} in Theorem~\ref{thm:loss} when~$a>0$, with the analogous result obtained without splitting the kernels in~\eqref{eq:classical} when~$1\leq p\leq 2\leq q\leq\infty$. The decay rate in~\eqref{eq:classical} is worse than the one provided by~\eqref{eq:LpLqloss}, when both~$p<2<q$. The decay rate is the same if~$p=2$ and~$q\in[2,\infty)$ or~$p\in(1,2]$ and~$q=2$.
\end{remark}
\begin{remark}\label{rem:lossbad}
If condition~\eqref{eq:wavereg} is violated, one may modify Theorem~\ref{thm:loss}, taking
\[ a = n\left(\frac1p-\frac1q\right)+|\beta|+b+\sigma\ell. \]
However, now, following the proof of Theorem~\ref{thm:lowder}, one only gets
\[
\|m\|_{M_{p}^{q}} \leq \sum_{k=0}^{k_0(t)+1}\|m\psi_k\|_{M_{p}^{q}} \leq C\,\sum_{k=0}^{k_0(t)+1} 2^{k\sigma\left(n\max\left\{\left(\frac12-\frac1p\right),\left(\frac1q-\frac12\right)\right\}-1\right)} \approx t^{n\max\left\{\left(\frac12-\frac1p\right),\left(\frac1q-\frac12\right)\right\}-1}\,,
\]
where we used~$2^{k_0(t)+2}\approx t^{\frac1\sigma}$. In turn, this gives
\[ \| \partial_x^\beta (-\Delta)^{\frac{b}2}\partial_t^\ell  K_0(t,\cdot)\ast u_1\|_{L^q} \lesssim\, (1+t)^{n\max\left\{\left(\frac12-\frac1p\right),\left(\frac1q-\frac12\right)\right\}-\frac{1}{2\theta}\left(n\left(\frac1p-\frac1q\right)+|\beta|+b+\sigma\ell\right)}\,\|u_1\|_{L^p}, \qquad  \forall t\geq 0. \]
This estimate is far from being optimal. Indeed, the decay rate may be improved, at least in high space dimension, if we do not split the two kernels, but we treat them together. In other words, when~$1\leq p\leq q<2$ or~$2<p\leq q\leq\infty$, the idea that we proposed in our paper to split the kernels and treat them separately, is clearly valid only if we remain in the validity of the regularity for $\sigma$-evolution equations damping free, namely, if~\eqref{eq:wavereg} holds. The following example shows this limit of the technique of splitting kernels.
\end{remark}
\begin{example}
Let~$\sigma=\theta=2$, $p=q=1$, $b=0$, and~$2|\beta|+\ell\geq1$. Condition~\eqref{eq:wavereg} is verified for~$n=1$, hence we may apply Theorem~\ref{thm:loss}, but does not hold for~$n\geq2$. Therefore, using Theorem~\ref{thm:loss} for~$n=1$ and following Remark~\ref{rem:lossbad} for~$n\geq2$, we obtain
\[ \| \partial_x^\beta \partial_t^\ell K_0(t,\cdot)\ast u_1\|_{L^1} \lesssim\, \begin{cases}
(1+t)^{\frac{n+2}4-\frac{|\beta|}4-\frac\ell2}\,\|u_1\|_{L^1}, & \text{if~$n=1,2$,}\\
(1+t)^{\frac{n}2-\frac{|\beta|}4-\frac\ell2}\,\|u_1\|_{L^1}, & \text{if~$n\geq2$.}
\end{cases}\]
Comparing with the result in~\cite{DAGL}, where the estimate
\[ \| \partial_x^\beta \partial_t^\ell K_0(t,\cdot)\ast u_1\|_{L^1} \lesssim\, (1+t)^{\frac{n+2}4-\frac{|\beta|}4-\frac\ell2}\,\|u_1\|_{L^1}, \qquad  \forall t\geq 0, \]
is proved in any space dimension~$n\geq1$, we see that the decay rates are the same at~$n=1,2$, but the decay rate in~\cite{DAGL} is better, as expected, for any~$n\geq3$. In~\cite{DAGL}, the kernels are not split as in this paper, but they are treated together.
\end{example}

\section{High frequencies estimates ($\xii\geq N_\infty$)}\label{sec:high}

Dealing with derivatives of the solution at high frequencies is not difficult, so we include the derivatives from the beginning in our statement.

If~$\theta<\sigma$, a smoothing effect appears which, in particular, allow us to deal with higher derivatives of the solution and to get a $L^r-L^q$ estimate for any~$1\leq r\leq q\leq\infty$, if we ``pay'' a singularity at~$t=0$. The singularity at~$t=0$ is related to the fact that the smoothing effect requires some positive time to produces its effect. This phenomenon is analogous to what happens in the heat equation. In the limit case~$\theta=\sigma$, the smoothing effect only influences the time derivatives.
\begin{theorem}\label{thm:sing}
Let~$\sigma>1$ and assume a noneffective damping, that is, $\theta\in(\sigma/2,\sigma]$. Let~$1\leq r \leq q \leq \infty$, $\beta\in\N^n$, $\ell\in\N$, and~$b\geq0$. Define
\begin{equation}\label{eq:xreg}
a=n\left(\frac1r-\frac1q\right)+|\beta|+b.
\end{equation}
Then we have the following estimate
\begin{eqnarray}\label{eq:singLrLq}
\| \partial_x^\beta(-\Delta)^{\frac{b}2}\partial_t^\ell K_\infty(t,\cdot)\ast u_1\|_{L^q} \lesssim\,t^{-\delta}\,e^{-ct}\,\|u_1\|_{L^r}, \qquad  \forall t>0,
\end{eqnarray}
where:
\begin{itemize}
\item if~$\theta<\sigma$ and~$a\geq2\theta$, then
\[ \delta = \ell+\frac{a-2\theta}{2(\sigma-\theta)}, \]
if~$(r,q)\in(1,\infty)$, whereas~$\delta$ may be any positive number verifying
\[ \delta>\ell+\frac{a-2\theta}{2(\sigma-\theta)}, \]
if~$r=1$ or~$q=\infty$;
\item if~$\theta\leq\sigma$ and~$a\leq2\theta$, then
\[ \delta = \left( \ell-1+\frac{a}{2\theta}\right)_+, \]
if~$(r,q)\in(1,\infty)$, whereas~$\delta$ may be any nonnegative number verifying
\[ \delta>\ell-1+\frac{a}{2\theta}, \]
if~$r=1$ or~$q=\infty$.
\end{itemize}
\end{theorem}
If we are interested in non-singular estimates, it is sufficient to take~$\delta=0$ in Theorem~\ref{thm:sing}, and we obtain the following immediate.
\begin{corollary}\label{cor:high}
Let~$\sigma>1$ and assume a noneffective damping, that is, $\theta\in(\sigma/2,\sigma]$. Assume that~$\beta\in\N^n$ and~$b\geq0$ verify~$|\beta|+b<2\theta$. Let~$1\leq r \leq q \leq \infty$, be such that
\begin{equation}\label{eq:genregLr}
n\left(\frac1r-\frac1q\right)\leq 2\theta-|\beta|-b,
\end{equation}
if~$r,q\in(1,\infty)$, or
\[ n\left(\frac1r-\frac1q\right)< 2\theta-|\beta|-b, \]
if~$r=1$ or~$q=\infty$.  Then we have the following estimate
\begin{eqnarray}\label{eq:LrLq}
\| \partial_x^\beta(-\Delta)^{\frac{b}2} K_\infty(t,\cdot)\ast u_1 \|_{L^q} \lesssim\, e^{-ct}\,\|u_1\|_{L^r}, \qquad  \forall t\geq 0.
\end{eqnarray}
Now let~$q\in(1,\infty)$. Then we have the following $L^q-L^q$ estimates
\begin{align}\label{eq:derLqLq}
\| \partial_x^\beta(-\Delta)^{\frac{b}2} K_\infty(t,\cdot)\ast u_1\|_{L^q} & \lesssim\, e^{-ct}\,\|u_1\|_{L^q}, \quad |\beta|+b=2\theta,\qquad \forall t\geq 0,\\
\| \partial_t K_\infty(t,\cdot)\ast u_1\|_{L^q} & \lesssim\, e^{-ct}\,\|u_1\|_{L^q}, \qquad \forall t\geq 0.
\end{align}
\end{corollary}
In order to prove Theorem~\ref{thm:sing}, we first derive $L^q-L^q$ estimates, with~$q\in(1,\infty)$ by using Mikhlin-H\"ormander theorem, then we use Hardy-Littlewood-Sobolev Theorem~\ref{thm:HLS} to obtain $L^r-L^q$ estimates, provided that~$(r,q)\in(1,\infty)$, as a corollary. To deal with the difficult case of $r=1$ or~$q=\infty$ in the estimates, we prove that
\[ \| \partial_x^\beta(-\Delta)^{\frac{b}2}\partial_t^\ell K_\infty(t,\cdot)\|_{L^\eta} \lesssim t^{-\delta}\,e^{-ct},\]
for~$\eta\in[1,\infty]$, and then we apply Young inequality.

We recall that
\[ \partial_t^\ell \hat K_\infty(t,\cdot) = \varphi_\infty\,\frac{\lambda_+^\ell}{\lambda_+-\lambda_-}\,e^{\lambda_+t} - \varphi_\infty\,\frac{\lambda_-^\ell}{\lambda_+-\lambda_-}\,e^{\lambda_-t}\,,\]
where
\[ \lambda_\pm = \frac12\,\xii^{2\theta}\Big(-1\pm\sqrt{1-4\xii^{-(4\theta-2\sigma)}}\,\Big). \]
In particular,
\[ \lambda_-\sim -\xii^{2\theta},\quad \lambda_+\sim -\xii^{2\sigma-2\theta},\qquad \text{as~$\xii\to\infty$.} \]
Therefore,
\begin{equation}\label{eq:basichigh}
\big|\partial_\xi^\gamma \big( (i\xi)^\beta\,\xii^b\,\partial_t^\ell \hat K_\infty(t,\xi)\big) \big| \lesssim \xii^{|\beta|+b-2\theta-|\gamma|}\,\Big( \xii^{(2\sigma-2\theta)\ell}\,e^{-ct\xii^{2\sigma-2\theta}} + \xii^{2\theta\ell}\,e^{-ct\xii^{2\theta}} \Big),
\end{equation}
for any~$\xii\geq N_\infty$.

Multiplying and dividing by~$t^\delta$, and using that
\[ t^\delta\xii^{\delta\kappa}\,e^{-\frac{c}2t\xii^\kappa} \]
is bounded for any~$\delta\geq0$, whereas~$e^{-\frac{c}2t\xii^\kappa}\leq e^{-c_1t}$, by virtue of~$\xii\geq N_\infty$, we obtain
\begin{equation}\label{eq:singhigh}
\big|\partial_\xi^\gamma \big( (i\xi)^\beta\,\xii^b\,\partial_t^\ell \hat K_\infty(t,\xi)\big) \big| \lesssim \xii^{|\beta|+b-2\theta-|\gamma|}\,\Big( \xii^{(2\sigma-2\theta)(\ell-\delta)}+ \xii^{2\theta(\ell-\delta)}\Big)\,t^{-\delta}\,e^{-ct}.
\end{equation}
We are now ready to prove our statements.
\begin{lemma}\label{lem:MH}
Let~$\sigma>1$ and assume a noneffective damping, that is, $\theta\in(\sigma/2,\sigma]$. Let~$q\in(1,\infty)$, $\beta\in\N^n$, $\ell\in\N$, and~$b\geq0$. Then we have the following estimate
\begin{eqnarray}\label{eq:singLqLq}
\| \partial_x^\beta(-\Delta)^{\frac{b}2}\partial_t^\ell K_\infty(t,\cdot)\ast u_1\|_{L^q} \lesssim\,t^{-\delta}\,e^{-ct}\,\|u_1\|_{L^q}, \qquad  \forall t>0,
\end{eqnarray}
where:
\begin{itemize}
\item if~$\theta<\sigma$ and~$|\beta|+b\geq2\theta$, then
\[ \delta = \ell+\frac{|\beta|+b-2\theta}{2(\sigma-\theta)}; \]
\item if~$\theta\leq\sigma$ and~$|\beta|+b\leq2\theta$, then
\[ \delta = \left( \ell-1+\frac{|\beta|+b}{2\theta}\right)_+. \]
\end{itemize}

\end{lemma}
\begin{proof}
By Mikhlin-H\"ormander theorem, the statement follow, since, by~\eqref{eq:singhigh} and by the definition of~$\delta$, we obtain
\[ \big|\partial_\xi^\gamma \big( (i\xi)^\beta\,\xii^b\,\partial_t^\ell \hat K_\infty(t,\xi)\big) \big| \lesssim \xii^{-|\gamma|}\,t^{-\delta}\,e^{-ct}. \]
Indeed, the quantity
\[ \xii^{|\beta|+b-2\theta}\,\Big( \xii^{(2\sigma-2\theta)(\ell-\delta)}+ \xii^{2\theta(\ell-\delta)}\Big) \]
is bounded for~$\xii\geq N_\infty$ if, and only if,
\[ |\beta|+b-2\theta \leq \begin{cases}
(2\sigma-2\theta)(\ell-\delta) & \text{if~$|\beta|+b-2\theta\geq0$,}\\
2\theta(\ell-\delta) & \text{if~$|\beta|+b-2\theta\leq0$.}
\end{cases} \]
\end{proof}
By Hardy-Littlewood-Sobolev Theorem~\ref{thm:HLS}, we get the proof of Theorem~\ref{thm:sing}, exception given for the cases~$r=1$ or~$q=\infty$.
\begin{proof}[Proof of Theorem~\ref{thm:sing} when~$r,q\in(1,\infty)$]
Let us define
\[ \kappa=n\left(\frac1r-\frac1q\right). \]
If~$p,q\in(1,\infty)$, we write
\[ K_\infty(t,\cdot)\ast u_1= (-\Delta)^{\frac\kappa2}\,K_\infty(t,\cdot) \ast (I_\kappa u_1), \]
so that, applying Lemma \ref{lem:MH} we obtain
\[ \| \partial_x^\beta(-\Delta)^{\frac{b}2}\partial_t^\ell K_\infty(t,\cdot)\ast u_1\|_{L^q} = \| \partial_x^\beta(-\Delta)^{\frac{b+\kappa}2}\partial_t^\ell K_\infty(t,\cdot)\ast (I_\kappa u_1)\|_{L^q} \lesssim\,t^{-\delta}\,e^{-ct}\,\|I_\kappa u_1\|_{L^r}, \]
and the proof follows by Hardy-Littlewood-Sobolev Theorem~\ref{thm:HLS}:
\[ \|I_\kappa u_1\|_{L^q}\leq C\,\|u_1\|_{L^r}. \]
\end{proof}
The proof of Theorem~\ref{thm:sing} follows from the following.
\begin{lemma}\label{lem:highL1}
Let~$\sigma>1$ and assume a noneffective damping, that is, $\theta\in(\sigma/2,\sigma]$. Let~$\beta\in\N^n$, $\ell\in\N$, and~$b\geq0$. Then
\begin{eqnarray}\label{eq:highL1}
\| \partial_x^\beta(-\Delta)^{\frac{b}2}\partial_t^\ell K_\infty(t,\cdot)\|_{L^\eta} \lesssim\,t^{-\delta}\,e^{-ct}\qquad  \forall t>0,
\end{eqnarray}
for any~$\eta\in[1,\infty]$, where:
\begin{itemize}
\item if~$\theta<\sigma$ and~$n(1-1/\eta)+|\beta|+b\geq2\sigma$, then~$\delta$ may be any nonnegative number verifying
\[ \delta>\ell+\frac{n\left(1-\frac1\eta\right)+|\beta|+b-2\theta}{2(\sigma-\theta)}; \]
\item if~$\theta\leq\sigma$ and~$n(1-1/\eta)+|\beta|+b\leq2\sigma$, then~$\delta$ may be any nonnegative number verifying
\[ \delta>\ell-1+\frac{n\left(1-\frac1\eta\right)+|\beta|+b}{2\theta}. \]
\end{itemize}
\end{lemma}
\begin{proof}
We only prove the Lemma for~$\eta=1$, being the other cases easier and analogous. Similarly to the proof of Lemma~\ref{lem:MH}, by~\eqref{eq:singhigh} and by the definition of~$\delta$, we obtain
\[ \big|\partial_\xi^\gamma \big( (i\xi)^\beta\,\xii^b\,\partial_t^\ell \hat K_\infty(t,\xi)\big) \big| \lesssim \xii^{-|\gamma|-\delta_1}\,t^{-\delta}\,e^{-ct}, \]
for some~$\delta_1\in(0,1)$. For the sake of brevity, let
\[ f(t,\xi)=t^{\delta}\,e^{ct}\,(i\xi)^\beta\,\xii^b\,\partial_t^\ell \hat K_\infty(t,\xi). \]
Proceeding as we do in Lemma~\ref{lem:parts}, integrating by parts~$n+1$ times, we easily get
\[
t^{\delta}\,e^{ct}\,|\partial_x^\beta(-\Delta)^{\frac{b}2}\partial_t^\ell K_\infty(t,x)|
    \lesssim |x|^{-n-1}\,\sum_{|\gamma|=n+1} \int_{\xii\geq N_\infty} |\partial_\xi^\gamma f(t,\xi)|\,d\xi
    \lesssim |x|^{-n-1}\,\int_{\xii\geq N_\infty} \xii^{-n-1-\delta_1} \,d\xi,
\]
with the latter integral being convergent. On the other hand, integrating first~$n-1$ times, splitting in two integrals and integrating by parts one more time in one of the two integrals (similarly to what we do in the proof of Lemma~\ref{lem:parts}), we obtain
\begin{align*}
& t^{\delta}\,e^{ct}\,|\partial_x^\beta(-\Delta)^{\frac{b}2}\partial_t^\ell K_\infty(t,x)|\\
    & \qquad \lesssim |x|^{-(n-1)}\,\sum_{|\gamma|=n-1} \int_{N_\infty\leq\xii\leq|x|^{-1}} |\partial_\xi^\gamma f(t,\xi)|\,d\xi\\
    & \qquad \qquad + |x|^{-n}\,\sum_{j=1}^n\sum_{|\gamma|=n-1} \int_{|x|^{-1}\leq \xii} |\partial_{\xi_j}\partial_\xi^\gamma f(t,\xi)|\,d\xi
    + |x|^{-n}\,\sum_{|\gamma|=n-1} \int_{\xii=|x|^{-1}} |\partial_\xi^\gamma f(t,\xi)|\,d\xi \\
    & \qquad \lesssim |x|^{-(n-1)}\, \int_{N_\infty\leq\xii\leq|x|^{-1}} \xii^{-(n+\delta_1-1)}\,d\xi \\
    & \qquad \qquad + |x|^{-n}\,\int_{|x|^{-1}\leq \xii} \xii^{-(n+\delta_1)}\,d\xi + |x|^{-n}\, \int_{\xii=|x|^{-1}} \xii^{-(n+\delta_1-1)}\,d\xi\\
    & \qquad \lesssim |x|^{-(n-\delta_1)}.
\end{align*}
Since we proved that
\[ t^{\delta}\,e^{ct}\,|\partial_x^\beta(-\Delta)^{\frac{b}2}\partial_t^\ell K_\infty(t,\cdot)| \lesssim |x|^{-(n-\delta_1)}(1+|x|)^{-1-\delta_1}, \]
for some~$\delta_1\in(0,1)$, our claim follows.
\end{proof}
%

\section{Proof of the global existence results}\label{nonlinear}

By Duhamel's principle, a function $u\in X$, where~$X$ is a suitable space, is a solution to~\eqref{eq:CPu} or~\eqref{eq:CPut} in~$X$ if, and only if, it satisfies the equality
\begin{equation}\label{eq:fixedpoint}
u(t,x) =  u^{\mathrm{lin}}(t,x) + \int_0^t K(t-s, x) \ast_{(x)} \, f(u(s,x),u_t(s,x))\, ds\,, \qquad \text{in~$X$,}
\end{equation}
where $f(u,u_t)=|u|^{1+\alpha}$ or $f(u,u_t)=|u_t|^{1+\alpha}$ and
\[ u^{\mathrm{lin}} (t,x) \doteq  K(t,x) \ast_{(x)} u_1(x) \,, \]
is the solution to the linear Cauchy problem~\eqref{eq:CPlin}.

The proof of our global existence results is based on the following scheme. We  define an appropriate data function space $\mathcal A $ and a space for solutions $X(T)$, equipped with a norm induced by some of the decay estimates we obtained for~$u^\lin$, assuming data in~$\mathcal A$. In particular, we look for the estimate
\begin{equation}\label{eq:ubasic}
\|u^{\mathrm{lin}}\|_{X(T)} \leq C\,\|(u_0, u_1)\|_{\mathcal A},
\end{equation}
with~$C$ independent of~$T$. We define the operator~$F$ such that, for any~$u\in X$,
\begin{equation}\label{eq:F}
Fu(t,x) \doteq \int_0^t K(t-s,x)\ast_{(x)} f(u(s,x), u_t(s,x))\, ds\,,
\end{equation}
then we prove the estimates
\begin{align}
\label{eq:well}
\|Fu\|_{X}
    & \leq C\|u\|_{X}^{1+\alpha}\,, \\
\label{eq:contraction}
\|Fu-Fv\|_{X}
    & \leq C\|u-v\|_{X} \bigl(\|u\|_{X}^{\alpha}+\|v\|_{X}^{\alpha}\bigr)\,,
\end{align}
with~$C$ independent of~$T$. By standard arguments, since $u^{\mathrm{lin}}$ satisfies~\eqref{eq:ubasic} and~$\alpha>0$, from~\eqref{eq:well} it follows that~$u^{\mathrm{lin}}+F$ maps balls of~$X$ into balls of~$X$, for small data in~$\mathcal{A}$, and that estimates \eqref{eq:well}-\eqref{eq:contraction} lead to the existence of a unique solution to~\eqref{eq:fixedpoint}, that is, $u=u^{\mathrm{lin}}+Fu$, satisfying~\eqref{eq:ubasic}. We simultaneously gain a local and a global existence result.

The information that~$u\in X$ plays a fundamental role to estimate~$f(u(s,\cdot),u_t(s,\cdot))$ in suitable norms. We will employ the following well-known result.
\begin{lemma}\label{lem:integral}
Let~$\nu<1<\mu$. Then it holds
\[
\int_0^t (t-s)^{-\nu}\,\,(1+s)^{-\mu}\,ds \lesssim (1+t)^{-\nu}.
\]
\end{lemma}
Lemma~\ref{lem:integral} has been proved in many different versions by many authors. One earlier version of this lemma goes back to~\cite{Segal}.

\subsection{Proof of Theorem~\ref{thnonlinear}}

In order to prove Theorem~\ref{thnonlinear}, for any~$T>0$, we fix the initial data space to be
\[ \mathcal A = L^1\cap L^\eta,\qquad \eta=\max\{2, n/(2\theta)\}, \]
and we introduce the solution space
\[ X(T)= \mathcal C ([0,T], H^{\sigma}\cap L^\infty) \cap \mathcal C^1([0,T],L^2), \]
equipped with norm
\begin{align}
\nonumber
\|u\|_{X(T)}
    & = \max_{t\in[0,T]} \big( M[u](t) + (1+t)^{-1+\frac{n}\sigma\left(1-\frac1{1+\alpha}\right)} \|u(t,\cdot)\|_{L^{1+\alpha}}\big),
\intertext{where}
\nonumber
M[u](t)
    & = (1+t)^{\frac{n}{4\theta}}\,\|(-\Delta)^{\frac{\sigma}{2}} u(t,\cdot)\|_{L^2} + \|u_t(t,\cdot)\|_{L^2}\big) + (1+t)^{\frac{n}\sigma-1}\|u(t,\cdot)\|_{L^\infty} \\
\label{eq:Mu}
    & \qquad + \|u(t,\cdot)\|_{L^2} \times \begin{cases}
    (1+t)^{-1+\frac{n}{2\sigma}} & \text{if~$n<2\sigma$,} \\
    (\log(e+t))^{-1} & \text{if~$n=2\sigma$,}\\
    (1+t)^{\frac{n-2\sigma}{4\theta}} & \text{if~$n>2\sigma$.}
    \end{cases}
\end{align}
We first prove~\eqref{eq:ubasic}, that is, $M[u^\lin](t)\leq C$, where~$C$ is independent of~$t$.
\begin{lemma}\label{lem:lin}
Let~$u_1\in\mathcal A$. Assume that~$n\leq\bar n(\sigma)$. Then~$u^\lin\in X(T)$ and~\eqref{eq:ubasic} holds, with~$C>0$ independent of~$T>0$.
\end{lemma}
\begin{proof}
Let $q=2$, $\beta=0$, and~$(b,\ell)=(\sigma,0), (0,1)$. If we apply Theorem~\ref{thm:loss} with~$p=1$ together with Corollary~\ref{cor:high} with~$r=2$, we obtain
\begin{equation}\label{eq:linen}
\|(-\Delta)^{\frac{\sigma}{2}} u^\lin(t,\cdot)\|_{L^2} +\|u_t^\lin(t,\cdot)\|_{L^2} \leq C\,(1+t)^{-\frac{n}{4\theta}}\,\|u_1\|_{L^1\cap L^2}.
\end{equation}
Indeed, $a=n/2>0$ in~\eqref{eq:loss} (due~$(p,q)=(1,2)$, the logarithmic loss may be removed, according to Remark~\ref{rem:loss}).

Now let~$q=2$, $|\beta|=b=\ell=0$. If~$n\leq 2\sigma$, applying Theorem~\ref{thlinearestimates} with~$p=1$ and~$r=2$, we obtain
\begin{equation*}
\|u^\lin(t,\cdot)\|_{L^2}  \leq C\,\|u_1\|_{L^1\cap L^2}\times \begin{cases}
(1+t)^{1-\frac{n}{2\sigma}} & \text{if~$n<2\sigma$,}\\
\log(e+t) & \text{if~$n=2\sigma$.}
\end{cases}
\end{equation*}
If~$n>2\sigma$, then~$a=(n-2\sigma)/2$ in~\eqref{eq:loss}, and we get
\[ \|u^\lin(t,\cdot)\|_{L^2}  \leq C\,(1+t)^{-\frac{n-2\sigma}{4\theta}}\,\|u_1\|_{L^1\cap L^2}. \]
Thanks to the assumption~$1+\alpha>1+\alpha_0$, together with~$n\leq \bar n(\sigma)$, we have that~\eqref{eq:rangeboth} holds with~$p=1$ and~$q=1+\alpha,\infty$ (see Remark~\ref{rem:barn}).

Let~$q=\infty$ and~$|\beta|=b=\ell=0$. By applying Theorem~\ref{thlinearestimates} with~$p=1$ and~$r=\eta$, we immediately obtain
\begin{equation}\label{eq:lininf}
\|u^\lin(t,\cdot)\|_{L^\infty} \leq C\,(1+t)^{1-\frac{n}\sigma}\,\|u_1\|_{L^1\cap L^\eta}.
\end{equation}
Similarly, let~$q=1+\alpha$ and~$|\beta|=b=\ell=0$. By applying Theorem~\ref{thlinearestimates} with~$p=1$ and~$r=\min\{1+\alpha,\eta\}$, we get
\[ \|u^\lin(t,\cdot)\|_{L^{1+\alpha}} \leq C\,(1+t)^{1-\frac{n}\sigma\left(1-\frac1{1+\alpha}\right)}\,\|u_1\|_{L^1\cap L^\eta}. \]
This concludes the proof.
\end{proof}
We are now ready to prove Theorem~\ref{thnonlinear}.
\begin{proof}[Theorem~\ref{thnonlinear}]
In view of Lemma~\ref{lem:lin}, we shall prove only~\eqref{eq:well} and \eqref{eq:contraction}. We prove \eqref{eq:well}, omitting the proof of~\eqref{eq:contraction}, since it is analogous to the proof of~\eqref{eq:well}. Let~$u\in X(T)$. We split the integral in~\eqref{eq:F} in the intervals~$[0,t/2]$ and~$[t/2,t]$ and we estimate~$M[Fu](t)$.

Let $q=2$, $\beta=0$, and~$(b,\ell)=(\sigma,0), (0,1)$. By applying Theorem~\ref{thm:loss} with~$p=1$ in~$[0,t/2]$ and Theorem~\ref{thm:lowder} with~$p=2$ in~$[t/2,t]$, together with Corollary~\ref{cor:high} with~$r=2$, we obtain
\begin{align*}
& \int_0^t \big(\|(-\Delta)^{\frac{\sigma}{2}} K(t-s,\cdot)\ast_{(x)}|u(s,\cdot)|^{1+\alpha}\|_{L^2}+\|K_t(t-s,\cdot)\ast_{(x)}|u(s,\cdot)|^{1+\alpha}\|_{L^2}\big)\,ds \\
    & \qquad \lesssim \int_0^{t/2} (1+t-s)^{-\frac{n}{4\theta}}\,\big(\|u(s,\cdot)\|_{L^{1+\alpha}}^{1+\alpha}+\|u(s,\cdot)\|_{L^{2(1+\alpha)}}^{1+\alpha}\big)\,ds + \int_{t/2}^t \|u(s,\cdot)\|_{L^{2(1+\alpha)}}^{1+\alpha}\,ds.
\end{align*}
Due to~$u\in X(T)$, by interpolation of~$L^{1+\alpha}$ and~$L^\infty$, we know that
\[ \|u(s,\cdot)\|_{L^q}\leq (1+s)^{1-\frac{n}\sigma\left(1-\frac1q\right)}\,\|u\|_{X(T)},\quad q\in[1+\alpha,\infty]. \]
In particular,
\begin{align*}
\|u(s,\cdot)\|_{L^{1+\alpha}}
    & \leq (1+s)^{1-\frac{n}\sigma\left(1-\frac1{1+\alpha}\right)}\,\|u\|_{X(T)},\\
\|u(s,\cdot)\|_{L^{2(1+\alpha)}}
    & \leq (1+s)^{1-\frac{n}\sigma\left(1-\frac1{2(1+\alpha)}\right)}\,\|u\|_{X(T)}.
\end{align*}
By using that~$t-s\approx t$ for~$s\in[0,t/2]$ and~$s\approx t$ for~$s\in[t/2,t]$, we then obtain
\begin{align*}
& \int_0^t \big(\|(-\Delta)^{\frac{\sigma}{2}} K(t-s,\cdot)\ast_{(x)}|u(s,\cdot)|^{1+\alpha}\|_{L^2}+\|K_t(t-s,\cdot)\ast_{(x)}|u(s,\cdot)|^{1+\alpha}\|_{L^2}\big)\,ds \\
    & \qquad \lesssim \|u\|_{X(T)}^{1+\alpha}\,\int_0^{t/2} (1+t-s)^{-\frac{n}{4\theta}}\,(1+s)^{1+\alpha-\frac{n}\sigma\alpha}\,ds + \|u\|_{X(T)}^{1+\alpha}\,\int_{t/2}^t (1+s)^{1+\alpha-\frac{n}\sigma\alpha-\frac{n}{2\sigma}}\,ds \\
    & \qquad \lesssim \|u\|_{X(T)}^{1+\alpha}\,(1+t)^{-\frac{n}{4\theta}}\,\int_0^{t/2} (1+s)^{1+\alpha-\frac{n}\sigma\alpha}\,ds + \|u\|_{X(T)}^{1+\alpha}\,(1+t)^{1+\alpha-\frac{n}\sigma\alpha-\frac{n}{2\sigma}}\,\int_{t/2}^t 1\,ds \\
    & \qquad \lesssim \|u\|_{X(T)}^{1+\alpha}\,(1+t)^{-\frac{n}{4\theta}}\,,
\end{align*}
where we used that
\[ \frac{n}\sigma\alpha - \alpha -1 = \frac{n-\sigma}{\sigma}\,\alpha -1>\frac{n-\sigma}\sigma\alpha_0-1= 1, \]
for any~$\alpha>\alpha_0$, and that~$n/(2\sigma)>n/(4\theta)$. We proceed in a similar way for $\| Fu \|_{L^2}$. If~$n<2\sigma$, then we apply Theorem~\ref{thlinearestimates} with~$p=1$ and~$q=r=2$, obtaining
\begin{align*}
\int_0^t \|K(t-s,\cdot)\ast_{(x)}|u(s,\cdot)|^{1+\alpha}\|_{L^2}\,ds
    & \lesssim \int_0^t (1+t-s)^{1-\frac{n}{2\sigma}}\,\big(\|u(s,\cdot)\|_{L^{1+\alpha}}^{1+\alpha}+\|u(s,\cdot)\|_{L^{2(1+\alpha)}}^{1+\alpha}\big)\,ds\\
    & \lesssim \|u\|_{X(T)}^{1+\alpha}\,\int_0^t (1+t-s)^{1-\frac{n}{2\sigma}}\,(1+s)^{1+\alpha-\frac{n}\sigma\alpha}\,ds\\
    & \lesssim (1+t)^{1-\frac{n}{2\sigma}}\,\|u\|_{X(T)}^{1+\alpha},
\end{align*}
by applying Lemma~\ref{lem:integral} with
\[ \nu=\frac{n}{2\sigma}-1<0, \qquad \mu = \frac{n-\sigma}\sigma\alpha-1>\frac{n-\sigma}\sigma\alpha_0-1=1. \]
The proof for~$n=2\sigma$ is analogous, but a logarithmic term appears, whereas for~$n>2\sigma$ we use Theorem~\ref{thm:loss} with~$a=(n-2\sigma)/2$, obtaining
\begin{align*}
\int_0^t \|K(t-s,\cdot)\ast_{(x)}|u(s,\cdot)|^{1+\alpha}\|_{L^2}\,ds
    & \lesssim \int_0^t (1+t-s)^{-\frac{n-2\sigma}{4\theta}}\,\big(\|u(s,\cdot)\|_{L^{1+\alpha}}^{1+\alpha}+\|u(s,\cdot)\|_{L^{2(1+\alpha)}}^{1+\alpha}\big)\,ds\\
    & \lesssim \|u\|_{X(T)}^{1+\alpha}\,\int_0^t (1+t-s)^{-\frac{n-2\sigma}{4\theta}}\,(1+s)^{1+\alpha-\frac{n}\sigma\alpha}\,ds\\
    & \lesssim (1+t)^{-\frac{n-2\sigma}{4\theta}}\,\|u\|_{X(T)}^{1+\alpha},
\end{align*}
by applying Lemma~\ref{lem:integral} with
\[ \nu=\frac{n-2\sigma}{4\theta}<1, \qquad \mu = \frac{n-\sigma}\sigma\alpha-1>\frac{n-\sigma}\sigma\alpha_0-1=1. \]
Indeed, $\nu\leq (\sigma-1)/(4\theta)<1$ for any~$n\leq\bar n$, in view of Remark~\ref{rem:barn} and~$\sigma\leq2\theta$.

Now let~$q=\infty$, $r=\eta$, $b=|\beta|=\ell=0$. If~$n<2\sigma$, by applying Theorem~\ref{thlinearestimates} with~$p=1$ in~$[0,t]$, we get
\begin{align*}
& \int_0^t \|K(t-s,\cdot)\ast_{(x)}|u(s,\cdot)|^{1+\alpha}\|_{L^\infty}\,ds \\
    & \qquad \lesssim \int_0^t (1+t-s)^{1-\frac{n}{\sigma}}\,\big(\|u(s,\cdot)\|_{L^{1+\alpha}}^{1+\alpha}+\|u(s,\cdot)\|_{L^{\eta(1+\alpha)}}^{1+\alpha}\big)\,ds \\
    & \qquad \lesssim \|u\|_{X(T)}^{1+\alpha}\,\int_0^t (1+t-s)^{1-\frac{n}{\sigma}}\,(1+s)^{1+\alpha-\frac{n}\sigma\alpha}\,ds \\
    & \qquad \lesssim \|u\|_{X(T)}^{1+\alpha}\,(1+t)^{1-\frac{n}{\sigma}}\,,
\end{align*}
where we used~$\alpha>\alpha_0$ and Lemma~\ref{lem:integral}. On the other hand, if~$n\geq 2\sigma$, we may apply Theorem~\ref{thlinearestimates} with~$p=1$ in~$[0,t/2]$, and with~$p=n/(2\sigma)$ in~$[t/2,t]$. Indeed, condition~\eqref{eq:rangeboth} with~$p=n/(2\sigma)$ and~$q=\infty$ reads as
\[ \frac{n}\sigma \left(\frac{2\sigma}n -\frac1\infty \right) + n\left(\frac12-\frac{2\sigma}n\right) \leq 1,\quad \text{i.e.,}\qquad n \leq 4\sigma-2. \]
This latter inequality holds as a consequence of~$n\leq \bar n(\sigma)$ (see Remark~\ref{rem:barn}). In turn, we obtain:
\begin{align*}
& \int_0^t \|K(t-s,\cdot)\ast_{(x)}|u(s,\cdot)|^{1+\alpha}\|_{L^\infty}\,ds \\
    & \qquad \lesssim \int_0^{t/2} (1+t-s)^{1-\frac{n}{\sigma}}\,\big(\|u(s,\cdot)\|_{L^{1+\alpha}}^{1+\alpha}+\|u(s,\cdot)\|_{L^{\eta(1+\alpha)}}^{1+\alpha}\big)\,ds \\
    & \qquad \qquad + \int_{t/2}^t (1+t-s)^{-1}\,\log(e+t-s)\,\big(\|u(s,\cdot)\|_{L^{(1+\alpha)n/(2\sigma)}}^{1+\alpha}+\|u(s,\cdot)\|_{L^{\eta(1+\alpha)}}^{1+\alpha}\big)\,ds \\
    & \qquad \lesssim \|u\|_{X(T)}^{1+\alpha}\,\int_0^{t/2} (1+t-s)^{1-\frac{n}{\sigma}}\,(1+s)^{1+\alpha-\frac{n}\sigma\alpha}\,ds \\
    & \qquad \qquad + \|u\|_{X(T)}^{1+\alpha}\,\int_{t/2}^t (1+t-s)^{-1}\,\log(e+t-s)\,(1+s)^{-\frac{n-\sigma}\sigma(1+\alpha)+2}\,ds \\
    & \qquad \lesssim \|u\|_{X(T)}^{1+\alpha}\,(1+t)^{1-\frac{n}{\sigma}}\int_0^{t/2} (1+s)^{1+\alpha-\frac{n}\sigma\alpha}\,ds \\
    & \qquad \qquad + \|u\|_{X(T)}^{1+\alpha}\,(1+t)^{-\frac{n-\sigma}\sigma(1+\alpha)+2}\,\int_{t/2}^t (1+t-s)^{-1}\,\log(e+t-s)\,ds \\
    & \qquad \lesssim \|u\|_{X(T)}^{1+\alpha}\,(1+t)^{1-\frac{n}{\sigma}}\,,
\end{align*}
where we used once again that
\[ \frac{n-\sigma}\sigma\alpha-1>\frac{n-\sigma}\sigma\alpha_0-1=1, \]
for any~$\alpha>\alpha_0$. Similarly, we obtain
\[ \int_0^t \|K(t-s,\cdot)\ast_{(x)}|u(s,\cdot)|^{1+\alpha}\|_{L^{1+\alpha}}\,ds \lesssim \|u\|_{X(T)}^{1+\alpha}\,(1+t)^{1-\frac{n}{\sigma}\left(1-\frac1{1+\alpha}\right)}\,. \]
This concludes the proof.
\end{proof}

\subsection{Proof of Theorem~\ref{thnonlinear2}}

In order to prove Theorem~\ref{thnonlinear2}, for any~$T>0$, we fix the initial data space to be
\[ \mathcal A = L^1\cap L^{1+\alpha}, \]
and we introduce the solution space
\[ X(T)= \mathcal C ([0,T], H^{\sigma}) \cap \mathcal C^1([0,T],L^2\cap L^{1+\alpha}), \]
equipped with norm
\[ \|u\|_{X(T)} = \max_{t\in[0,T]} \big( M[u](t) + (1+t)^{\frac{n}\sigma\left(1-\frac1{1+\alpha}\right)} \|u_t(t,\cdot)\|_{L^{1+\alpha}}\big),\]
where~$M[u](t)$ is as in~\eqref{eq:Mu}. We remark that now the last term in~$M[u](t)$ is
\[ (1+t)^{-1+\frac{n}{2\sigma}}\|u(t,\cdot)\|_{L^2},\]
since~$n\leq \sigma-2<2\sigma$.

We first prove~\eqref{eq:ubasic}, that is, $M[u^\lin](t)\leq C$, where~$C$ is independent of~$t$.
\begin{lemma}\label{lem:lin2}
Let~$u_1\in\mathcal A$. Assume that~$\sigma\geq3$ and~$n\leq \sigma-2$. Then~$u^\lin\in X(T)$ and~\eqref{eq:ubasic} holds, with~$C>0$ independent of~$T>0$.
\end{lemma}
\begin{proof}
Following as in the proof of Lemma~ \ref{lem:lin}, we derive~\eqref{eq:linen} and
\[ \|u^\lin(t,\cdot)\|_{L^2}  \leq C\,(1+t)^{1-\frac{n}{2\sigma}}\,\|u_1\|_{L^1\cap L^2}.\]
Since~\eqref{eq:rangeboth} holds with~$p=1$ and~$q=\infty$, as a consequence of~$n\leq \sigma-2$, we also derive~\eqref{eq:lininf}.

By using~$1+\alpha>1+\alpha_1$, together with~$n\leq \sigma-2$, we get~\eqref{eq:rangeENboth} with~$p=1$ and~$q=1+\alpha$. Let~$q=1+\alpha$, $|\beta|=b=0$ and~$\ell=1$. By applying Theorem~\ref{thm:lowder} with~$p=1$ together with Corollary~\ref{cor:high} with~$r=1+\alpha$, we obtain
\[ \|u^\lin_t(t,\cdot)\|_{L^{1+\alpha}} \leq C\,(1+t)^{-\frac{n}\sigma\left(1-\frac1{1+\alpha}\right)}\,\|u_1\|_{L^1\cap L^{1+\alpha}}. \]
This concludes the proof.
\end{proof}
We are now ready to prove Theorem~\ref{thnonlinear2}.
\begin{proof}[Theorem~\ref{thnonlinear2}]
In view of Lemma~\ref{lem:lin2}, we shall prove only~\eqref{eq:well} and \eqref{eq:contraction}. We prove \eqref{eq:well}, omitting the proof of~\eqref{eq:contraction}, since it is analogous to the proof of~\eqref{eq:well}. Let~$u\in X(T)$.

We may follow the proof of Theorem~\ref{thnonlinear}, but we now use singular estimates at high frequencies, to avoid to deal with~$\|u_t(s,\cdot)\|_{L^\infty}$.

Let $q=2$, $\beta=0$, and~$(b,\ell)=(\sigma,0), (0,1)$. We fix~$\delta\in(0,1)$, such that~$\delta>n/(4\theta)$. We may take this choice of such~$\delta<1$, due to~$n\leq \sigma-2<4\theta$. By applying Theorem~\ref{thm:loss} with~$p=1$, together with Theorem~\ref{thm:sing} with~$r=1$, we obtain
\begin{align*}
& \int_0^t \big(\|(-\Delta)^{\frac{\sigma}{2}} K(t-s,\cdot)\ast_{(x)}|u_t(s,\cdot)|^{1+\alpha}\|_{L^2}+\|K_t(t-s,\cdot)\ast_{(x)}|u_t(s,\cdot)|^{1+\alpha}\|_{L^2}\big)\,ds \\
    & \qquad \lesssim \int_0^t \big((1+t-s)^{-\frac{n}{4\theta}}\,+ (t-s)^{-\delta} e^{-c(t-s)}\big)\,\|u_t(s,\cdot)\|_{L^{1+\alpha}}^{1+\alpha}\,ds.
\end{align*}
Due to~$u\in X(T)$, we know that
\[ \|u_t(s,\cdot)\|_{L^{1+\alpha}}\leq (1+s)^{-\frac{n}\sigma\left(1-\frac1{1+\alpha}\right)}\,\|u\|_{X(T)};\]
hence, we get
\begin{align*}
& \int_0^t \big(\|(-\Delta)^{\frac{\sigma}{2}} K(t-s,\cdot)\ast_{(x)}|u_t(s,\cdot)|^{1+\alpha}\|_{L^2}+\|K_t(t-s,\cdot)\ast_{(x)}|u_t(s,\cdot)|^{1+\alpha}\|_{L^2}\big)\,ds \\
    & \qquad \lesssim \|u\|_{X(T)}\,\int_0^t \big((1+t-s)^{-\frac{n}{4\theta}} + (t-s)^{-\delta} e^{-c(t-s)}\big)\,(1+s)^{-\frac{n}\sigma\alpha}\,ds\\
    & \qquad \lesssim (1+t)^{-\frac{n}{4\theta}}\,\|u\|_{X(T)}\,,
\end{align*}
where we used~$n\leq \sigma-2<4\theta$, and
\[ \frac{n}\sigma\alpha>\frac{n}\sigma\alpha_1=1, \]
together with Lemma~\ref{lem:integral} to obtain
\[ \int_0^t (1+t-s)^{-\frac{n}{4\theta}}\,(1+s)^{-\frac{n}\sigma\alpha}\,ds \lesssim (1+t)^{-\frac{n}{4\theta}}\,, \]
whereas we estimate
\[ \int_0^{t/2} (t-s)^{-\delta} e^{-c(t-s)}\,(1+s)^{-\frac{n}\sigma\alpha}\,ds \lesssim e^{-ct/2}, \]
and
\[ \int_{t/2}^t (t-s)^{-\delta} e^{-c(t-s)}\,(1+s)^{-\frac{n}\sigma\alpha}\,ds \lesssim (1+t)^{-\frac{n}\sigma\alpha} \leq (1+t)^{-\frac{n}{4\theta}}, \]
thanks again to~$\alpha>\alpha_1$. To deal with the term $\| Fu(t,\cdot)\|_{L^2}$ we do not need the singular estimates. Indeed, being~$n\leq4\theta$, we may take~$p=r=1$ in Theorem~\ref{thlinearestimates}, and get
\[ \int_0^t\|K(t-s,\cdot)\ast_{(x)}|u_t(s,\cdot)|^{1+\alpha}\|_{L^2}\,ds\lesssim \|u\|_{X(T)}\,\int_0^t (1+t-s)^{1-\frac{n}{2\sigma}}\,(1+s)^{-\frac{n}\sigma\alpha}\,ds\lesssim (1+t)^{1-\frac{n}{2\sigma}}\,\|u\|_{X(T)}\,,\]
for any~$\alpha>\alpha_1$.

Similarly, let~$q=\infty$, $b=|\beta|=\ell=0$. Due to~$n\leq \sigma-2<2\theta$, by applying Theorem~\ref{thlinearestimates} with~$p=r=1$, we obtain
\begin{align*}
\int_0^t \|K(t-s,\cdot)\ast_{(x)}|u_t(s,\cdot)|^{1+\alpha}\|_{L^\infty}\,ds
    & \lesssim \int_0^t (1+t-s)^{1-\frac{n}{\sigma}} \|u_t(s,\cdot)\|_{L^{1+\alpha}}^{1+\alpha}\,ds \\
    & \lesssim \|u\|_{X(T)}^{1+\alpha}\,\int_0^t (1+t-s)^{1-\frac{n}{\sigma}}\,(1+s)^{-\frac{n}\sigma\alpha}\,ds \\
    & \lesssim \|u\|_{X(T)}^{1+\alpha}\,(1+t)^{1-\frac{n}{\sigma}}\,,
\end{align*}
by Lemma~\ref{lem:integral}, due to~$1-n/\sigma>0$ and~$\alpha>\alpha_1$.

Finally, let~$q=1+\alpha$, $b=|\beta|=0$, $\ell=1$. Thanks again to~$n\leq \sigma-2<2\theta$, we may fix~$\delta\in(0,1)$ such that~$\delta>n/(2\theta)$. By applying Theorem~\ref{thm:lowder} with~$p=1$ and Theorem~\ref{thm:sing} with~$r=1$, we derive
\begin{align*}
& \int_0^t \|K_t(t-s,\cdot)\ast_{(x)}|u_t(s,\cdot)|^{1+\alpha}\|_{L^{1+\alpha}}\,ds \\
    & \qquad \lesssim \int_0^t \big( (1+t-s)^{-\frac{n}{\sigma}\left(1-\frac1{1+\alpha}\right)} + (t-s)^{-\delta}\,e^{-c(t-s)} \big) \,\|u_t(s,\cdot)\|_{L^{1+\alpha}}^{1+\alpha}\,ds \\
    & \qquad \lesssim \|u\|_{X(T)}^{1+\alpha}\,\int_0^t \big( (1+t-s)^{-\frac{n}{\sigma}\left(1-\frac1{1+\alpha}\right)} + (t-s)^{-\delta}\,e^{-c(t-s)} \big)\,(1+s)^{-\frac{n}\sigma\alpha}\,ds \\
    & \qquad \lesssim \|u\|_{X(T)}^{1+\alpha}\,(1+t)^{-\frac{n}{\sigma}\left(1-\frac1{1+\alpha}\right)}\,,
\end{align*}
where we used~$n/\sigma\leq(\sigma-2)/\sigma<1$, $\alpha>\alpha_1$ and Lemma~\ref{lem:integral}.

This concludes the proof.
\end{proof}

\appendix

\section{Multiplier theorems}\label{sec:Appendix}
In this appendix we collect several results employed in the paper to prove that a function is a multiplier in~$M_p^q$, basing on suitable estimates for the function and its derivatives.

A key result for multipliers in~$M_p$ with~$p\in(1,\infty)$ is the Mikhlin-H\"ormander multiplier theorem.
\begin{theorem}\label{MH}
Let $1<p<\infty$ and $k=[n/2]+1$. Suppose that $m \in \mathcal C^{k}(\R^{n}\backslash \left\{ 0 \right\})$ and
\begin{equation}\nonumber
\left| \partial_\xi^{\gamma}m(\xi)\right| \leq C\, |\xi|^{-|\gamma|}, \quad |\gamma|\leq k.
\end{equation}
Then $m \in M_p$.
\end{theorem}
By Young inequality, $\hat K\in M_p^q$ if~$K\in L^r$, with~$1-1/r=1/p-1/q$.

Mikhlin-H\"ormander multiplier theorem is often used together with Hardy-Littlewood-Sobolev theorem for the Riesz potential~$I_a$.
\begin{theorem}\label{thm:HLS}
Let~$a\in(0,n)$ and~$p\in(1,n/a)$. Then~$\xii^{-a} \in M_p^q(\R^n)$, that is, $I_a\in L_p^q(\R^n)$, with~$q$ obtained by
\[ \frac1q = \frac1p-\frac{a}n. \]
\end{theorem}
A function~$m$ is a multiplier in~$M_1$ if~$\mathfrak{F}^{-1}m\in L^1$. In particular, this is true if~$m\in H^N$, for some~$N>n/2$. The following result, which also provides an estimate for~$\|\mathfrak{F}^{-1}m\|_{L^1}$, is of great interest for us.
\begin{theorem}\label{Th: Bernstein}
Let $n\geq 1$ and $N>n/2$. Assume that $m\in H^N$, then $\mathscr{F}^{-1}m\in L^1$ and there exists a constant $C>0$ such that
\[ \norma{\mathscr{F}^{-1}m}_{L^1}\leq C\, \norma{m}_{L^2}^{1-\frac{n}{2N}}\norma{D^Nm}_{L^2}^{\frac{n}{2N}}.\]
\end{theorem}
\begin{proof}
Let $f=\mathscr{F}^{-1}m$. If $f=0$ then the statement is trivial. Otherwise, let
\[ r=\norma{m}_{L^2}^{-\frac{1}{N}}\norma{D^Nm}_{L^2}^{\frac{1}{N}}.\]
Then, using H\"older's inequality, we get
\begin{align*}
\norma{f}_{L^1}
    & =\int_{|x|\leq r} |f(x)|\, dx+ \int_{|x|\geq r} |x|^{-N}|x|^N|f(x)|\, dx \\
    & \lesssim r^{\frac{n}2}\norma{f}_{L^2}+ r^{-N+\frac{n}{2}}\norma{|\cdot |^N f}_{L^2} \approx r^{\frac{n}2}\norma{m}_{L^2}+ r^{-N+\frac{n}2}\norma{D^N m}_{L^2},
\end{align*}
and the proof follows.
\end{proof}
%

The estimates provided by Theorem~\ref{Th: Bernstein} will be used together with the estimates for~$\|\mathfrak{F}^{-1}m\|_{L^\infty}$, provided by the following application of Littman's lemma, based on stationary phase methods (see, for instance, \cite{Pecher}).
\begin{lemma}\label{ThmLittmanlemmapecher}
Let us consider for $\tau \geq \tau_0$, $\tau_0$ is a large positive number, the oscillating integral
\[  F^{-1}_{\eta\rightarrow x}\big(e^{-i\tau \omega(\eta)} v(\eta)\big).\]
The amplitude function $v=v(\eta)$ is supposed to belong to $\mathcal C_c^\infty(\mathbb{R}^n)$ with support in $\{\eta \in \mathbb{R}^n: |\eta| \in [\frac{1}{2},2]\}$. The function $\omega=\omega(\eta)$ is $C^\infty$ in a neighborhood of the support of $v$. Moreover, the Hessian $H_\omega(\eta)$ is nonsingular, i.e., $\det H_\omega(\eta)\neq0$, on the support of $v$. Then the following $L^\infty-L^\infty$ estimate holds:
\[   \big\|F^{-1}_{\eta\rightarrow x}\big(e^{-i\tau \omega(\eta)} v(\eta)\big)\|_{L^\infty(\mathbb{R}^n_x)} \leq C(1+\tau)^{-\frac{n}{2}} \sum_{|\beta| \leq L} \|D^\beta_\eta v(\eta)\|_{L^\infty(\mathbb{R}^n_\eta)}, \]
where $L$ is a suitable entire number.
\end{lemma}
In  Proposition 2.5 of \cite{SS} one can find a simple proof of Lemma \ref{ThmLittmanlemmapecher}, from which it is easy to check that the statement remains valid whenever~$\omega$ and~$v$ depend on some parameter~$t$, provided that~$|\det H_\omega(t,\eta)|\geq c>0$, with~$c$ uniform with respect to~$t$. This property in our paper appears in~\eqref{eq:Hessian}.

Another strategy to derive multiplier estimates, showing that~$\mathfrak{F}^{-1} m\in L^r$, for some~$r\in[1,\infty]$ and then applying Young inequality, is based on the use of integration by parts in the formula for the inverse Fourier transform~$m$ to derive pointwise estimates for~$\mathfrak{F}^{-1} m$. This strategy is particularly effective if~$m$ is compactly supported, as in the following Lemma~\ref{lem:parts}, or if~$m$ vanishes in a neighborhood of the origin as in Lemma~\ref{lem:highL1}.
\begin{lemma}\label{lem:parts}
Assume that~$f\in\mathcal C_c^\kappa(\R^n)$ for some~$\kappa\geq0$, integer, and that it verifies the estimates
\[ \forall |\alpha|\leq \kappa: \qquad |\partial_x^\alpha f(\xi)| \leq C\,\xii^{-a},\]
for some~$a<n$. Then~$g=\mathfrak{F}^{-1}f$ satisfies the estimate~$|g(x)|\leq C'\,(1+|x|)^{-\kappa}$.

Moreover, if~$f\in\mathcal C_c^{\kappa+1}$ and
\[ \forall |\alpha|= \kappa+1: \qquad |\partial_x^\alpha f(\xi)| \leq C\,\xii^{-a_1},\]
for some~$a_1\in [n,n+1)$, then
\[ |g(x)| \leq \begin{cases}
C'\,(1+|x|)^{-\kappa-(n-a)} & \text{if~$a>a_1-1$,} \\
C'\,(1+|x|)^{-\kappa-(n+1-a_1)} & \text{if~$a\leq a_1-1$ and~$a_1\in(n,n+1)$,} \\
C'\,(1+|x|)^{-\kappa-1}\,\log (e+|x|) & \text{if~$a\leq n-1$ and~$a_1=n$.}
\end{cases}
\]
\end{lemma}
\begin{proof}
Due to~$a<n$, by the compact support of~$f$, we obtain~$\partial_\xi^\alpha f\in L^1$ for~$|\alpha|\leq\kappa$, so that~$(1+|x|)^k\,g\in\mathcal C_0$. This proves the first part of the statement.

Thanks to
\begin{equation}\label{eq:parts}
e^{ix\xi} = -\sum_{j=1}^n \frac{ix_j}{|x|^2}\,\partial_{\xi_j} e^{ix\xi},
\end{equation}
after integrating by parts~$\kappa$ times, we may write
\[ g(x) = (2\pi)^{-n}\,\int_{\R^n} e^{ix\xi} f(\xi)\,d\xi = (2\pi)^{-n} |x|^{-\kappa} \sum_{|\gamma|=\kappa} c_\gamma \,\int_{\R^n} e^{ix\xi} \partial_\xi^\gamma f(\xi)\,d\xi,\]
where we used that~$f$ is compactly supported. We now split the integral in two parts and we apply one extra step of integration by parts in the latter integral:
\begin{align*}
\int_{\R^n} e^{ix\xi} \partial_\xi^\gamma f(\xi)\,d\xi
    & = \int_{\xii\leq |x|^{-1}} e^{ix\xi} \partial_\xi^\gamma f(\xi)\,d\xi - \sum_{j=1}^n \frac{ix_j}{|x|^2}\, \int_{\xii=|x|^{-1}} e^{ix\xi} \partial_\xi^\gamma f(\xi)\,d\xi \\
    & \qquad  + \sum_{j=1}^n \frac{ix_j}{|x|^2}\, \int_{\xii\geq|x|^{-1}} e^{ix\xi} \partial{\xi_j} \partial_\xi^\gamma f(\xi)\,d\xi\,.
\end{align*}
Let~$M>|x|^{-1}$ be such that~$\supp f\subset \{\xii<M\}$. Then we may estimate
\begin{align*}
&\int_{\xii\leq |x|^{-1}} |\partial_\xi^\gamma f(\xi)|\,d\xi \leq C\,\int_{\xii\leq |x|^{-1}} \xii^{-a}\,d\xi = C_1\,|x|^{-(n-a)}, \\
& |x|^{-1}\,\int_{\xii=|x|^{-1}} |\partial_\xi^\gamma f(\xi)|\,d\xi = |x|^{-1}\,\int_{\xii=|x|^{-1}} \xii^{-a} \,d\xi = C_2\,|x|^{-(n-a)},\\
& |x|^{-1}\, \int_{\xii\geq |x|^{-1}} |\partial{\xi_j} \partial_\xi^\gamma f(\xi)|\,d\xi = |x|^{-1}\, \int_{|x|^{-1}\leq\xii\leq M} \xii^{-a_1}\,d\xi =\begin{cases}
C_3\,|x|^{-(n+1-a_1)} & \text{if~$a_1>n$,}\\
C_3\,|x|^{-1}\log (M|x|) & \text{if~$a_1=n$.}
\end{cases}
\end{align*}
As a consequence,
\[ |g(x)| \leq \begin{cases}
C'\,|x|^{-\kappa-(n-a)} & \text{if~$a<a_1-1$,} \\
C'\,|x|^{-\kappa-(n+1-a_1)} & \text{if~$a\geq a_1-1$ and~$a_1\in(n,n+1)$,} \\
C'\,|x|^{-\kappa-1}\,\log (e+|x|) & \text{if~$a\geq n-1$ and~$a_1=n$.}
\end{cases}
\]
This concludes the proof.
\end{proof}
\begin{remark}\label{rem:parts}
In particular, if
\[ \forall \alpha: \qquad |\partial_x^\alpha f(\xi)| \leq C_\alpha\,(1+\xii^{d-|\alpha|}),\]
for some~$d>-n$, then we may apply Lemma~\ref{lem:parts} with~$\kappa = n-1+\ceil{d}$, i.e., $\kappa$ is the biggest integer verifying~$\kappa<n+d$. Setting~$a=\kappa-d$ (we notice that~$a\in[n-1,n)$) and~$a_1=a+1$, we get
\[ |g(x)| \leq \begin{cases}
C'\,(1+|x|)^{-n-d} & \text{if~$d$ is not integer,} \\
C'\,(1+|x|)^{-n-d}\,\log (e+|x|) & \text{if~$d$ is integer.}
\end{cases}
\]
As a consequence, $g\in L^1\cap L^\infty$ if~$d>0$, and~$g\in L^r$ for any~$r>n/(n+d)$, if~$d\leq0$.
\end{remark}

\section*{Acknowledgments}

This paper has been realized during the stay of the second author to the Department of Mathematics of University of Bari in the period September-December 2019, supported by the ``Visiting professor program'' of University of Bari. The first author is member of the Gruppo Nazionale per l'Analisi Matematica, la Probabilit\`a e le loro Applicazioni (GNAMPA) of the Istituto Nazionale di Alta Matematica (INdAM).


\end{document}